\documentclass[12pt,reqno]{amsart}
\pdfoutput=1 % for arXiv to recognize this as pdfLaTeX
\usepackage[english]{babel}
\usepackage{graphicx}
\usepackage{latexsym,amssymb,amsfonts,amsmath}
\usepackage{enumitem}
\usepackage{url}
\makeatletter
\@namedef{subjclassname@2020}{%
  \textup{2020} Mathematics Subject Classification}
\makeatother

\newtheorem{theorem}{Theorem}
\newtheorem{corollary}[theorem]{Corollary}
\newtheorem{example}[theorem]{Example}
\newtheorem{lemma}[theorem]{Lemma}
\newtheorem{definition}[theorem]{Definition}

\newtheorem{fact}[theorem]{Fact}
\newtheorem{proposition}[theorem]{Proposition}

\newtheorem{remark}[theorem]{Remark}

\newcommand\cb[1]{\mathbf{#1}} % or \pmb

\newcommand\Latt{\mathcal{L}}

\newcommand\dsup{\sup\nolimits^{\scriptstyle\uparrow}}

\newcommand\dcup{\bigcup\nolimits^{\scriptstyle\uparrow}}

\newcommand\eqdef{\mathrel{\buildrel \text{def}\over=}}
\newcommand\nat{\mathbb{N}}
\newcommand\Z{\mathbb{Z}}
\newcommand\rat{\mathbb{Q}}
\newcommand\D{\mathbb{Q}_2}
\newcommand{\real}{\mathbb{R}}
\newcommand\Rp{\real_+}
\newcommand\creal{\overline{\real}_+}
\newcommand\diff{\smallsetminus}
\newcommand{\Pervcat}{\cb{Perv}}
\newcommand{\BPervcat}{\cb{BPerv}}

\newcommand{\Topcat}{\cb{Top}}
\newcommand{\omegaTopcat}{\omega\cb{Top}}
\newcommand{\Mes}{\cb{Mes}}

\newcommand{\Lcont}{\mathfrak L}
\newcommand\limp{\Rightarrow}
\newcommand\Borel[1]{\mathcal B ({#1})}

\numberwithin{equation}{section}

\title[Radon-Nikod\'ym]{A Radon-Nikod\'ym Theorem for Valuations}

\author{Jean Goubault-Larrecq\\
  Universit{\'e} Paris-Saclay, CNRS, ENS Paris-Saclay, Laboratoire
  M{\'e}thodes Formelles,\\
  4, avenue des sciences,
  91190, Gif-sur-Yvette, France.\\
  E-mail: \texttt{goubault@lsv.fr}}

\date{}

\keywords{Measure, valuation, Radon-Nikod\'ym derivative, density function}
\subjclass[2010]{Primary
  28C15; % Set functions and measures on topological spaces (regularity
  Secondary
  60B05.   %	Probability measures on topological spaces
}
\begin{document}

%%%%% To ease editing, for IMPAN journals add:

\baselineskip=17pt

%%%%%%%%%%%%%%%%

\begin{abstract}
  We enquire under which conditions, given two $\sigma$-finite,
  $\omega$-continuous valuations $\nu$ and $\mu$, $\nu$ has density
  with respect to $\mu$.  The answer is that $\nu$ has to be
  absolutely continuous with respect to $\mu$, plus a certain Hahn
  decomposition property, which happens to be always true for
  measures.
\end{abstract}

% MSc 2020:
% 46B22: Radon-Nikodým, Kreĭn-Milman and related properties [See also
% 46G10]
% non: 46, c'est functional analysis

\maketitle
\markboth{J. Goubault-Larrecq}{Radon-Nikod\'ym for valuations}

\noindent
\begin{minipage}{0.25\linewidth}
  \includegraphics[scale=0.2]{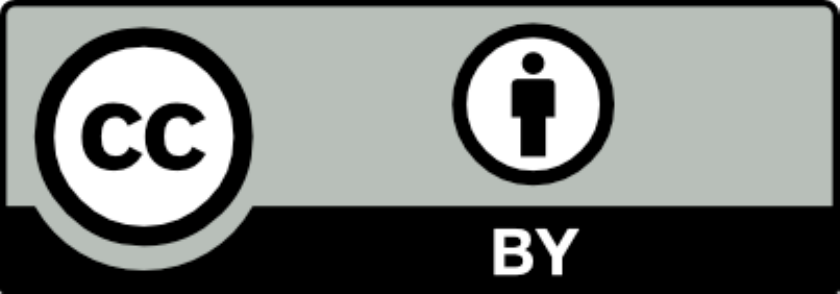}
\end{minipage}
\begin{minipage}{0.74\linewidth}
  \scriptsize
  For the purpose of Open Access, a CC-BY public copyright licence has
  been applied by the authors to the present document and will be
  applied to all subsequent versions up to the Author Accepted
  Manuscript arising from this submission.
\end{minipage}

\section{Introduction}

In its simplest form, the Radon-Nikod\'ym theorem
\cite{Radon:nikodym,Nikodym:radon} states that a $\sigma$-finite
measure $\nu$ has a measurable density with respect to a
$\sigma$-finite measure $\mu$ if and only if $\nu$ is absolutely
continuous with respect to $\mu$.  The purpose of this paper is to
investigate a similar question in the larger setting of
$\omega$-continuous valuations, a setting which encompasses both
measures and the continuous valuations used in the semantics of
probabilistic programming languages \cite{jones89,Jones:proba}.

Probably the distinguishing feature of valuations compared to measures
is that they give mass to sets forming a collection that is not
necessarily closed under complements: a lattice of subsets of
valuations, a topology for continuous valuations, and what we call an
$\omega$-topology for $\omega$-continuous valuations.

Sets equipped with such collection of sets are Pervin spaces,
topological spaces, and what we call $\omega$-topological spaces
respectively.  They form categories $\Pervcat$, $\Topcat$ and
$\omegaTopcat$ respectively.

As we will see, the question of the existence of density maps is more
about the category in which the density maps should reside, not so
much about the distinction between valuations and mesaures.  Indeed,
on sufficiently nice topological spaces, continuous valuations and
measures are essentially the same thing, and therefore
\emph{measurable} density maps will exist under the familiar
assumptions of the classical Radon-Nikod\'ym theorem.  This will also
entail that they do not exist in general as morphisms in $\Topcat$ or
$\omegaTopcat$, as we will see in Section~\ref{sec:absol-cont-not}.

Hence some additional assumptions are needed to ensure that density
maps exist in the relevant categories, and it is the purpose of this
paper to identify them.

\emph{Outline.}  We give brief preliminaries in
Section~\ref{sec:preliminaries}, and we develop the theory of
valuations, including $\omega$-continuous valuations, measures and
continuous valuations, in Section~\ref{sec:valuations-measures}.  We
develop necessary conditions for density maps to exist in
Section~\ref{sec:density-maps}, and we show that they are sufficient
in Section~\ref{sec:exist-dens-maps}.  Our final result includes the
classical Radon-Nikod\'ym theorem as a special case.

% Our motivation comes from theoretical computer science, but there are
% surely many other reasons one could be interested in extensions of the
% Radon-Nikod\'ym theorem.  In recent years, there has been a growing
% interest in stochastic programming languages, and their semantics.
% Among the features of those programming languages is the possibility
% of conditioning.  Exact conditioning (see \cite{SS:prob:exact} for
% example) can be realized in semantics that take place in Markov
% categories with conditionals
% \cite{Golubtsov:axiom:info,Golubtsov:kleisli:info,Golubtsov:info:cat,GM:add:info,CJ:disint,Fritz:cond},
% and conditionals in the Kleisli category of the Giry monad
% \cite{giry82} are simply Radon-Nikod\'ym derivatives.  There are other
% natural categories in which one can give semantics to probabilistic
% programs, and notably the Kleisli category of the continuous valuation
% monad on the category $\Dcpo$ of directed-complete partial orders
% (dcpos) \cite{jones89,Jones:proba}, and one may wonder whether those
% categories, too, are Markov categories with conditionals.  As we will
% see, the answer is no.  But we will have to investigate before we can
% reach this conclusion; in the process, we will see what can be saved.

\section{Preliminaries}
\label{sec:preliminaries}

We assume some basic knowledge about topology \cite{JGL-topology} and
about measure theory \cite{Billingsley:probmes}.  We will need the
following from domain theory \cite{JGL-topology,GHKLMS:contlatt}.

A \emph{directed} family $D$ in a poset $P$ is a non-empty family such
that any two elements of $D$ have a common upper bound in $D$.  A
\emph{dcpo} (short for directed-complete partial order) is a poset in
which every directed family has a supremum.  We write $\dsup D$, or
$\dsup_{i \in I} x_i$ if $D = {(x_i)}_{i \in I}$, for directed
suprema.  We also write $\dcup$ for directed union.  An
\emph{$\omega$cpo} is defined similarly, except that we only require
the existence of suprema of \emph{monotone sequences}
${(x_n)}_{n \in \nat}$ (namely,
$x_0 \leq x_1 \leq \cdots \leq x_n \leq \cdots$) instead of directed
families.

A function $f \colon X \to Y$ between dcpos is \emph{Scott-continuous}
if and only it is monotonic (order-preserving) and preserves suprema
of directed sets, namely
$\dsup_{i \in I} f (x_i) = f (\dsup_{i \in I} x_i)$ for every directed
family ${(x_i)}_{i \in I}$.  It is \emph{$\omega$-continuous} if and
only if it is monotonic and preserves suprema of monotone sequences.

The \emph{Scott topology} on a dcpo has as open sets those subsets $U$
that are upwards-closed (if $x \in U$ and $x \leq y$ then $y \in U$)
and such that every directed family $D$ such that $\dsup D \in U$
intersects $U$.  The Scott-continuous maps are exactly the continuous
maps with respect to the Scott topologies.

\section{Valuations and measures}
\label{sec:valuations-measures}

As our general setting, we will consider pairs $(X, \Latt)$ where $X$
is a set and $\Latt$ is a lattice of subsets, namely a family of
subsets of $X$ that is closed under finite intersections and finite
unions.  In particular, the empty set and $X$ belong to $\Latt$.

We retrieve topological spaces by requiring that $\Latt$ be closed
under arbitrary unions; or just under directed unions.  Indeed,
the union of any family ${(U_i)}_{i \in I}$ of subsets of $X$ is equal
to the directed supremum $\dcup_{J \text{ finite }\subseteq I}
\bigcup_{j \in J} U_j$.

We will call \emph{$\omega$-topology} on $X$ any lattice of subsets
$\Latt$ that is at the same time an $\omega$cpo under inclusion.  Then
$(X, \Latt)$ is an \emph{$\omega$-topological space}.  It is
equivalent to require that $\Latt$ be closed under countable unions,
since the union of any countable family ${(U_n)}_{n \in \nat}$ of
elements of $\Latt$ is the union
$\dcup_{n \in \nat} \bigcup_{i=0}^n U_i$ of a chain of elements of
$\Latt$.

% A \emph{$\sigma$-lattice} is a lattice of subsets $\Latt$ that is
% closed under both countable unions and countable intersections.

A lattice of subsets $\Latt$ that is closed under complements is a
\emph{algebra of subsets}, and an $\omega$-topology
% , or a $\sigma$-lattice,
$\Latt$ that is closed under complements is the same thing as a
$\sigma$-algebra.  Then $(X, \Latt)$ is called a \emph{measurable
  space}.

There are categories $\Pervcat$, $\BPervcat$, $\Topcat$,
$\omegaTopcat$ %,$\sPervcat$
and $\Mes$ whose objects are pairs $(X, \Latt)$ where $\Latt$ is a
lattice of subsets, resp.\ an algebra of subsets, resp.\ a topology,
resp.\ an $\omega$-topology, % resp.\ a $\sigma$-lattice,
resp.\ a $\sigma$-algebra.  In each case, the morphisms
$f \colon (X, \Latt) \to (Y, \Latt')$ are the maps $f \colon X \to Y$
such that $f^{-1} (V) \in \Latt$ for every $V \in \Latt'$.  They are
called \emph{continuous} maps on $\Topcat$, and \emph{measurable} maps
on $\Mes$.  The categories $\Pervcat$ and $\BPervcat$ are the
categories of \emph{Pervin spaces} and \emph{Boolean Pervin spaces}
respectively \cite[Section~3.1]{Pin:pervin}.  % , and we decide to call
% $\sPervcat$ the category of \emph{$\sigma$-Pervin spaces}.
Those
categories are all full subcategories of $\Pervcat$.

Let $\creal$ be the dcpo of extended non-negative real numbers
$\Rp \cup \{\infty\}$, with the usual ordering $\leq$ extended by the
stipulation that $r \leq \infty$ for every $r \in \creal$.  We will
always equip $\creal$ with the Scott topology of $\leq$, making it an
object of all the categories mentioned above.  The open subsets of
that Scott topology are the half-open intervals $]t, \infty]$,
$t \in \Rp$, plus $\creal$ and $\emptyset$.

We write $\Lcont (X, \Latt)$ for the set of morphisms from
$(X, \Latt)$ to $\creal$ (implicitly equipped with its Scott
topology), in any full subcategory of $\Pervcat$ containing $\creal$.
In other words, the elements $h$ of $\Lcont (X, \Latt)$ are the
functions $h \colon X \to \creal$ such that $h^{-1} (]t, \infty]) \in
\Latt$ for every $t \in \Rp$.

% In general, we will call \emph{admissible} any replete, full
% subcategory $\catc$ of $\Pervcat$ containing $\creal$.%   An
% % \emph{admissible} category is a replete, full subcategory of
% % $\omegaTopcat$ containing $\creal$.
% We write $\Lcont (X, \Latt)$ for the set of morphisms from $(X,
% \Latt)$ to $\creal$
% We write $\Lcont ((X, \Latt), (Y, \Latt'))$ for the set of
% $\catc$-morphisms from $(X, \Latt)$ to $(Y, \Latt')$, and
% $\Lcont (X, \Latt)$ for $\Lcont ((X, \Latt), \creal)$.  The notation
% does not mention $\catc$: by fullness, those sets are independent of
% the %weakly
% admissible category $\catc$ that $(X, \Latt)$ and $(Y, \Latt')$ live
% in.

When $(X, \Latt)$ is a measurable space, $\Lcont (X, \Latt)$ is the
set of all measurable maps from $(X, \Latt)$ to $\creal$ with its
usual Borel $\sigma$-algebra, generated by the
intervals.  % ; indeed, a map
% $h \colon (X, \Latt) \to \creal$ is measurable, where $\creal$ is
% given its usual Borel $\sigma$-algebra, generated by all intervals, if
% and only if $h^{-1} (]t, \infty]) \in \Latt$ for every $t \in \Rp$.
This is because one can write any interval as a Boolean combination of
intervals of the form $]t, \infty]$.  When $(X, \Latt)$ is a
topological space, $\Lcont (X, \Latt)$ is known as the set of
\emph{lower semicontinuous maps} from $(X, \Latt)$ to $\creal$.

If $\Latt$ is an $\omega$-topology (resp., a topology), then
$\Lcont (X, \Latt)$ is an $\omega$cpo (resp., a dcpo) under the
pointwise ordering defined by $h \leq h'$ if and only if
$h (x) \leq h' (x)$ for every $x \in X$; additionally, suprema of
monotone sequences (resp., directed suprema) are computed pointwise:
$(\dsup_{i \in I} h_i) (x) = \dsup_{i \in I} (h_i (x))$.  In order to
see this, it suffices to show that $\dsup_{i \in I} (h_i (x))$ is in
$\Lcont (X, \Latt)$; and the inverse image of $]t, \infty]$ under that
map is $\dcup_{i \in I} h_i^{-1} (]t, \infty])$, since $]t, \infty]$
is Scott-open.%   Similarly, $\omegaTopcat (X, \Latt)$
% % , $\sPervcat (X, \Latt)$
% and $\Mes (X, \Latt)$ are $\omega$-dcpo, where suprema of monotone
% sequences are computed pointwise.

Given any Pervin space $(X, \Latt)$, a \emph{valuation} $\nu$ on
$(X, \Latt)$ is a map $\nu \colon \Latt \to \creal$ that is:
\begin{itemize}
\item \emph{strict}: $\nu (\emptyset) = 0$;
\item \emph{monotonic}: $U \subseteq V$ implies
  $\nu (U) \leq \nu (V)$;
\item \emph{modular}: for all $U, V \in \Latt$, $\nu (U) + \nu (V) =
  \nu (U \cup V) + \nu (U \cap V)$.
\end{itemize}
A \emph{continuous valuation} is a valuation that is Scott-continuous,
and an \emph{$\omega$-continuous valuation} is a valuation that is
$\omega$-continuous.

Continuous valuations have been the cornerstone of the
domain-theoretic semantics of probabilistic languages since Claire
Jones' PhD thesis \cite{jones89,Jones:proba}, and had first been
studied by Nait Saheb-Djahromi \cite{saheb-djahromi:meas}.  The
concept of valuation is older, and dates back to Smiley
\cite{smiley44}, Horn and Tarski \cite{HornTarski48:ext}, and Pettis
\cite{Pettis:ext}, at least; see \cite{KR}.

An $\omega$-continuous valuation on a measurable space $(X, \Latt)$ is
a \emph{measure}.  Measures are usually defined as $\sigma$-additive
maps $\nu \colon \Latt \to \creal$, but the two definitions are
equivalent.  Let us recall that $\nu \colon \Latt \to \creal$ is
\emph{additive} (where $\Latt$ is any lattice of subsets) if and only
if $\nu (\emptyset)=0$ and $\nu (U \cup V) = \nu (U) + \nu (V)$ for
all pairs of two disjoint sets $U, V \in \Latt$, and
\emph{$\sigma$-additive} (where $\Latt$ is any $\omega$-topology) if
and only if $\nu (\bigcup_{i \in I} U_n) = \sum_{i \in I} \nu (U_i)$
for every countable family ${(U_i)}_{i \in I}$ of pairwise disjoint
elements $U_i$ of $\Latt$.  The equivalence of $\omega$-continuous
valuations and $\sigma$-additive maps on $\sigma$-algebras follows
from the following facts.
\begin{itemize}
\item If $\Latt$ is an algebra of subsets, then the additive maps
  $\nu \colon \Latt \to \creal$ are exactly the valuations on
  $(X, \Latt)$.  Indeed, if $\nu$ is additive, then strictness is
  clear, monotonicity follows from the fact that if $U \subseteq V$,
  then $\nu (V) = \nu (V \diff U) + \nu (U) \geq \nu (U)$, and
  modularity from
  $\nu (U) + \nu (V) = \nu (U \diff V) + \nu (U \cap V) + \nu (V) =
  \nu (U \cap V) + \nu (U \cup V)$.  Conversely, any valuation $\nu$
  is additive, since if $U$ and $V$ are disjoint, then $\nu (U \cup V)
  = \nu (U \cup V) + \nu (U \cap V) = \nu (U) + \nu (V)$.
\item If $\Latt$ is an $\omega$-topology, then the $\sigma$-additive
  maps are exactly the $\omega$-continuous, additive maps.  This
  follows from the fact that every countably infinite union
  $\bigcup_{n \in \nat} U_n$ can be written as
  $\dcup_{n \in \nat} \bigcup_{i=0}^n U_i$, plus additivity.
\end{itemize}

Addition is Scott-continuous on $\creal$, and it follows that
valuations on $(X, \Latt)$ form a dcpo under the \emph{stochastic
  ordering}, defined by $\mu \leq \nu$ if and only if
$\mu (U) \leq \nu (U)$ for every $U \in \Latt$; directed suprema are
computed pointwise:
$(\dsup_{i \in I} \nu_i) (U) = \dsup_{i \in I} (\nu_i (U))$.  The same
can be said for continuous valuations on a topological space, or for
$\omega$-continuous valuations on an $\omega$-topological space, hence
also for measures on a measurable space, since suprema commute.

% Let us fix a %weakly
% admissible category $\catc$.  
The simplest way to define a notion of integration is by the following
\emph{Choquet formula} \cite[Chapter~VII, Section~48.1,
p. 265]{Choquet:capacities}:
\begin{align}
  \label{eq:Choquet}
  \int_{x \in X} h (x) \;d\nu & \eqdef \int_0^{\infty} \nu (h^{-1} (]t,
                               \infty])) \;dt,
\end{align}
for every function $h \in \Lcont (X, \Latt)$, and for every valuation
$\nu$ on $(X, \Latt)$.  The integral on the right is an ordinary
improper Riemann integral, which is well-defined because the map
$t \mapsto \nu (h^{-1} (]t, \infty]))$ is antitonic (order-reversing).
Indeed, it is easy to see that, for any antitonic map
$f \colon \Rp \to \creal$, $\int_0^\infty f (t) \;dt$ is the supremum
of the monotone sequence of lower Darboux sums
$\sum_{k=1}^{N2^N} f (\frac k {2^N})$, $N \in \nat$.  This was already
observed in the proof of Lemma~4.2 of Regina Tix's master's thesis
\cite{Tix:bewertung}, which also contains the following statement; the
proof boils down to a familiar commutation of suprema.
\begin{fact}[Lemma~4.2, 3rd item, of \cite{Tix:bewertung}]
  \label{fact:R:cont}
  Riemann integration is Scott-continuous in the integrated antitonic
  map.  In particular, for any directed family ${(f_i)}_{i \in I}$
  (countable or not) of antitonic maps from $\Rp$ to $\creal$, in the
  pointwise ordering,
  $\int_0^\infty \dsup_{i \in I} f_i (t) \;dt = \dsup_{i \in I}
  \int_0^\infty f_i (t) \;dt$.
\end{fact}

Equation (\ref{eq:Choquet}) makes sense for more general set functions
$\nu$ than just valuations, but we will not make use of this.  We also
write $\int h \;d\nu$ for $\int_{x \in X} h (x) \;d\nu$.

We sum up the main properties of the Choquet integral in the following
proposition; $h$, $h'$ and $h_i$ stand for a arbitrary elements of
$\Lcont (X, \Latt)$, $\mu$, $\nu$ and $\nu_i$ for valuations on
$(X, \Latt)$, $a$ and $b$ are arbitrary elements of $\Rp$.  Addition
and multiplication on $\creal$ are defined in the obvious way, with
the caveat that $0.\infty = \infty.0 = 0$, so as to ensure that
multiplication, not just addition, is Scott-continuous.  On spaces of
$\creal$-valued maps and of valuations, addition and scalar
multiplication are defined pointwise.  The \emph{characteristic map}
$\chi_U \colon X \to \creal$ maps every $x \in U$ to $1$ and all other
points to $0$; $\chi_U$ is in $\Lcont (X, \Latt)$ if and only if
$U \in \Latt$.  The \emph{Dirac valuation} $\delta_x$ maps every
$U \in \Latt$ to $1$ if $x \in U$, to $0$ otherwise; namely,
$\delta_x (U) = \chi_U (x)$.  Given a morphism
$f \colon (X, \Latt) \to (Y, \Latt')$, the \emph{image valuation}
$f [\nu]$ of any valuation $\nu$ on $(X, \Latt)$ is defined by
$f [\nu] (V) \eqdef \nu (f^{-1} (V))$; this is a valuation, resp.\ an
$\omega$-continuous valuation, resp.\ a measure, resp.\ a continuous
valuation if $\nu$ is.
\begin{proposition}
  \label{prop:Choquet:props}
  Choquet integration is:
  \begin{enumerate}[label=(\roman*)]
  \item linear in the valuation: $\int h \;d(a\mu + b\nu) = a \int h
    \;d\mu + b \int h \;d\nu$;
  \item Scott-continuous in the valuation: $\int h \;d \sup_{i \in I}
    \nu_i = \dsup_{i \in I} \int h \;d\nu_i$ if ${(\nu_i)}_{i \in I}$ is
    directed;
  \item linear in the integrated function if $(X, \Latt)$ is an
    $\omega$-topological space and $\nu$ is an
    $\omega$-continuous valuation: $\int (a h +b h') \;d\nu = a \int h
    \;d\nu + b \int h' \;d\nu$;
  \item $\omega$-continuous in the integrated function if $(X, \Latt)$
    is an $\omega$-topological space and $\nu$ is $\omega$-continuous
    (in particular,
    $\int \dsup_{i \in \nat} h_i \;d\nu = \dsup_{i \in \nat} h_i \;d\nu$),
    and Scott-continuous if $(X, \Latt)$ is a topological space and
    $\nu$ is a continuous valuation (notably,
    $\int \dsup_{i \in I} h_i \;d\nu = \dsup_{i \in I} \int h_i \;d\nu$ if
    ${(h_i)}_{i \in I}$ is directed).
  \end{enumerate}
  Additionally,
  \begin{enumerate}[resume*]
  \item $\int \chi_U \;d\nu = \nu (U)$ for every $U \in \Latt$;
  \item $\int h \;d\delta_x = h (x)$ for every $x \in X$.
  % \item the \emph{change-of-variables} formula holds:
  %   \begin{align}
  %     \label{eq:chgvar}
  %     \int h \;df[\nu] & = \int (h \circ f) \;d\nu
  %   \end{align}
  %   for every morphism $f \colon (X, \Latt) \to (Y, \Latt')$, for
  %   every $h \in \Lcont (Y, \Latt')$, and for every valuation $\nu$ on
  %   $(X, \Latt)$.
  \end{enumerate}
\end{proposition}
\begin{proof}
  The argument follows classical lines, and most notably
  \cite[Section~4]{Tix:bewertung}.
  
  Item~$(i)$ follows from the fact that Riemann integration is itself
  linear, and $(ii)$ follows from Fact~\ref{fact:R:cont}; monotonicity
  is clear.  Item~$(v)$ follows from the fact that
  $(\dsup_{i \in I} h_i)^{-1} (]t, \infty]) = \dcup_{i \in I} h_i^{-1}
  (]t, \infty]$, $\nu$ is Scott-continuous and Fact~\ref{fact:R:cont}.
  Item $(iv)$ is proved similarly.  As far as $(v)$ is concerned, we
  have
  $\int \chi_U \;d\nu = \int_0^\infty \nu (\chi_U^{-1} (]t, \infty))
  \;dt = \int_0^\infty f (t) \;dt$ where $f$ maps every $t \in [0, 1[$
  to $\nu (U)$ and every $t \geq 1$ to $0$.  For $(vi)$,
  $\int h \;d\delta_x = \int_0^\infty \delta_x (h^{-1} (]t, \infty]))
  \;dt = \int_0^\infty g (t) \;dt$ where $g$ maps every $t < h (x)$ to
  $1$, and every $t \geq h (x)$ to $0$.  %   The change-of-variables
  % formula is an immediate consequence of Choquet's formula.
  The only
  tricky point is to show item~$(iii)$.

  First, we have
  $\int ah \;d\nu = \int_0^{\infty} \nu ({ah}^{-1} (]t, \infty])$.  If
  $a=0$, this is equal to $0 = a . \int h \;d\nu$.  If
  $a \neq 0, \infty$, this is equal to
  $\int_0^{\infty} \nu (h^{-1} (]t/a, \infty]) \;dt = \int_0^{\infty}
  \nu (h^{-1} (]u, \infty]) . a \;du = a \int_0^\infty \nu (h^{-1} (]u,
  \infty]) \;du = a \int h \;d\nu$.  Hence $\int ah \;d\nu = a \int h \;d\nu$
  for every $a \in \Rp$; this also holds when $a=\infty$ by $(iv)$,
  since $\infty = \dsup \nat$.  Hence it suffices to show that
  $\int (h+h') \;d\nu = \int h \;d\nu + \int h \;d\nu'$.

  We proceed in steps.  We fix $h$.  For every $\epsilon \in \Rp$, and
  for every $U \in \Latt$, we claim that:
  \begin{align}
    \label{eq:choquet:add}
    & \int_\epsilon^\infty \nu (h^{-1} (]t, \infty]) \cup (h^{-1}
      (]t-\epsilon, \infty]) \cap U)) \;dt \\
    \nonumber
    & = \int_\epsilon^\infty \nu
      (h^{-1} (]t, \infty])) \;dt + \int_0^\epsilon \nu (h^{-1} (]t,
      \infty]) \cap U) \;dt.
  \end{align}
  If
  $\int_\epsilon^\infty \nu (h^{-1} (]t, \infty]) \cap U) \;dt <
  \infty$, then we reason as follows.  By the modularity law, the fact
  that the intersection of $h^{-1} (]t, \infty])$ with
  $h^{-1} (]t-\epsilon, \infty]) \cap U$ simplifies to
  $h^{-1} (]t, \infty]) \cap U$, and the usual properties of Riemann
  integrals,
  \begin{align*}
    & \int_\epsilon^\infty \nu (h^{-1} (]t, \infty]) \cup (h^{-1}
      (]t-\epsilon, \infty]) \cap U)) \;dt \\
    & = \int_\epsilon^\infty \nu (h^{-1} (]t, \infty])) \;dt
      + \int_\epsilon^\infty \nu (h^{-1}
      (]t-\epsilon, \infty]) \cap U) \;dt \\
    & \qquad - \int_\epsilon^\infty \nu (h^{-1} (]t, \infty]) \cap U)
      \;dt \\
    & = \int_\epsilon^\infty \nu (h^{-1} (]t, \infty])) \;dt
      + \int_0^\infty \nu (h^{-1}
      (]t, \infty]) \cap U) \;dt \\
    & \qquad - \int_\epsilon^\infty \nu (h^{-1} (]t, \infty]) \cap U)
      \;dt \\
    & = \int_\epsilon^\infty \nu (h^{-1} (]t, \infty])) \;dt
      + \int_0^\epsilon \nu (h^{-1} (]t, \infty]) \cap U) \;dt.
  \end{align*}
  If
  $\int_\epsilon^\infty \nu (h^{-1} (]t, \infty]) \cap U) \;dt =
  \infty$, then since $h^{-1} (]t, \infty]) \cap U$ is included in
  %both
  $h^{-1} (]t, \infty]) \cup (h^{-1} (]t-\epsilon, \infty]) \cap U)$,
  % and in $h^{-1} (]t, \infty])$,
  both sides of (\ref{eq:choquet:add}) are
  equal to $\infty$.

  Now, $\int (h+\epsilon \chi_U) \;d\nu$ is equal to
  $\int_0^\infty (h + \epsilon \chi_U)^{-1} (]t, \infty]) \;dt$, and
  $(h + \epsilon \chi_U)^{-1} (]t, \infty])$ is equal to
  $h^{-1} (]t, \infty]) \cup U$ if $t < \epsilon$ and to
  $h^{-1} (]t, \infty]) \cup (h^{-1} (]t-\epsilon, \infty]) \cap U)$
  otherwise.  Therefore:
  \begin{align*}
    & \int (h+\epsilon \chi_U) \;d\nu \\
    & = \int_0^\epsilon \nu (h^{-1} (]t, \infty]) \cup U) \;dt
      + \int_\epsilon^\infty \nu (h^{-1} (]t, \infty]) \cup (h^{-1}
      (]t-\epsilon, \infty]) \cap U)) \;dt \\
    & = \int_0^\epsilon \nu (h^{-1} (]t, \infty]) \cup U) \;dt
      + \int_\epsilon^\infty \nu (h^{-1} (]t, \infty])) \;dt \\
    & \qquad + \int_0^\epsilon \nu (h^{-1} (]t, \infty]) \cap U) \;dt
    \qquad \text{by (\ref{eq:choquet:add})} \\
    & = \int_0^\epsilon \nu (h^{-1} (]t, \infty]) \;dt +
      \int_\epsilon^\infty \nu (h^{-1} (]t, \infty])) \;dt \\
    & \qquad + \int_0^\epsilon \nu (U) \;dt
    \qquad\qquad\qquad \text{by modularity of $\nu$ under the $\int_0^\epsilon$
      terms} \\
    & = \int h \;d\nu + \epsilon \nu (U).
  \end{align*}
  This being done, let a \emph{very simple function} be any map $h'$
  of the form $\epsilon \sum_{i=1}^n \chi_{U_i}$ where
  $\epsilon \in \Rp$ and $U_i \in \Latt$.  By induction on $n$, and
  using what we have just proved, we obtain that
  $\int (h+h') \;d\nu = \int h \;d\nu + \int h' \;d\nu$.
  % induction case: \int (h+h') \;d\nu
  % = \int (h+ \epsilon \sum_{i=1}^{n-1} \chi_{U_i}) \;d\nu + \epsilon
  % \nu (U_n)
  % and similarly for \int h \;d\nu + \int h' \;d\nu

  Finally, every $h' \in \Lcont (X, \Latt)$ is the supremum of the
  monotone sequence of very simple functions
  $h'_N \eqdef \frac 1 {2^N} \sum_{i=1}^{N2^N} \chi_{{h'}^{-1}
    (]i/2^N, \infty])}$, $N \in \nat$.  Then
  $\int (h+h') \;d\nu = \dsup_{N \in \nat} \int (h+h'_N) \;d\nu = \dsup_{N
    \in \nat} (\int h \;d\nu + \int h'_N \;d\nu) = \int h \;d\nu + \int h'
  \;d\nu$ by using $(iv)$.
\end{proof}

Property~$(iv)$ is usually called the monotone convergence theorem (or
the Beppo Levi theorem) when applied to measurable spaces and measures.

We will also use the following baby version of the Riesz
representation theorem.  A \emph{linear} map
$F \colon \Lcont (X, \Latt) \to \creal$ is one such that
$F (ah)=aF(h)$ for all $a \in \Rp$ (\emph{positive homogeneity}) and
$h \in \Lcont (X, \Latt)$ and $F (h+h')=F(h)+F(h')$ for all
$h, h' \in \Lcont (X, \Latt)$ (\emph{additivity}).  It is equivalent
to require $F (ah+bh') = a F (h) + b F (h')$ for all $a, b \in \Rp$
and $h, h' \in \Lcont (X, \Latt)$; if $F$ is $\omega$-continuous, then
this extends to the cases where $a$ or $b$ or both is equal to
$\infty$.

\begin{proposition}
  \label{prop:riesz}
  Let $(X, \Latt)$ be an $\omega$-topological space (resp., a
  topological space).  There is a one-to-one correspondence between
  $\omega$-continuous (resp., continuous) valuations $\nu$ on
  $(X, \Latt)$ and linear $\omega$-continuous (resp.,
  Scott-continuous) maps $F \colon \Lcont (X, \Latt) \to \creal$.  In
  one direction, $F (h) \eqdef \int h \;d\nu$, and in the other
  direction, $\nu (U) \eqdef F (\chi_U)$.
\end{proposition}
\begin{proof}
  We deal with the $\omega$-continuous case only, since the continuous
  case is similar.  The continuous case was also dealt with by Tix
  \cite[Satz~4.16]{Tix:bewertung}, using similar arguments.  Given an
  $\omega$-continuous valuation $\nu$, the map
  $F_\nu \colon h \mapsto \int h \;d\nu$ is $\omega$-continuous and
  linear by items~$(ii)$ and $(iv)$ of
  Proposition~\ref{prop:Choquet:props}.  Conversely, given an
  $\omega$-continuous linear map
  $F \colon \Lcont (X, \Latt) \to \creal$, we define
  $\nu_F (U) \eqdef F (\chi_U)$.  Then $\nu_F$ is strict since $F$
  maps the constant $0$ map to $0$ by positive homogeneity,
  $\omega$-continuous since $F$ is, and since the map
  $U \mapsto \chi_U$ is itself $\omega$-continuous, and modular
  because of the equality
  $\chi_U+\chi_V = \chi_{U \cup V} + \chi_{U \cap V}$ and the
  additivity of $F$.  We have $\nu_{F_\nu} = \nu$, because for every
  $U \in \Latt$,
  $\nu_{F_\nu} (U) = F_\nu (\chi_U) = \int \chi_U \;d\nu = \nu (U)$ by
  item~$(v)$ of Proposition~\ref{prop:Choquet:props}.  In order to
  show that $F_{\nu_F} = F$, we realize that
  $F_{\nu_F} (\chi_U) = \int \chi_U \;d\nu_F = \nu_F (U) = F (\chi_U)$
  by item~$(v)$ of Proposition~\ref{prop:Choquet:props}.  Then, by the
  linearity of the integral (item~$(iii)$), $F_{\nu_F} (h) = F (h)$
  for every very simple function (as introduced in the proof of
  Proposition~\ref{prop:Choquet:props}), and since every element of
  $\Lcont (X, \Latt)$ is a supremum of a monotone sequence of very
  simple functions, we conclude by the $\omega$-continuity of $F$ and
  of $F_{\nu_F}$ (item~$(iv)$) that $F_{\nu_F} = F$.
\end{proof}

\section{Density maps}
\label{sec:density-maps}

% One may define the valuation $g \cdot \mu$ obtained from a valuation
% $\mu$ and a density map $g \in \Lcont (X, \Latt)$ by setting
% $(g \cdot \mu) (U) \eqdef \int \chi_U. g \;d\mu$, but the following
% route will allow us to obtain the integration formula with respect to
% $g \cdot \mu$ for free.
\begin{lemma}
  \label{lemma:gmu}
  Let $(X, \Latt)$ be an $\omega$-topological space, let
  $g \in \Lcont (X, \Latt)$, and $\mu$ be an $\omega$-continuous
  valuation on $(X, \Latt)$.

  The map that sends every $h \in \Lcont (X, \Latt)$ to
  $\int h g \;d\mu$ is well-defined, linear and $\omega$-continuous.

  It is Scott-continuous provided that $\Latt$ is a topology and $\mu$
  is Scott-continuous.
\end{lemma}
\begin{proof}
  We must first show that the integral makes sense, namely that the
  product map $hg$ is in $\Lcont (X, \Latt)$.  The multiplication map
  $a, b \mapsto ab$ is Scott-continuous from $\creal \times \creal$ to
  $\creal$, hence, for every $t > 0$, $ab > t$ if and only if there
  are two rational numbers $p, q > 0$ such that $p>a$, $q>b$ and
  $pq>t$.  For every $t > 0$, $(hg)^{-1} (]t, \infty])$ is then equal
  to
  $\bigcup_{p, q \in \rat, pq>t} h^{-1} (]p, \infty]) \cap g^{-1} (]q,
  \infty])$.  That is an infinite countable union, hence it is in
  $\Latt$.  Therefore $hg$ is in $\Lcont (X, \Latt)$.

  Since product by $g$ is linear and
  $\omega$-continuous (even Scott-continuous), the remaining claims
  follow from items $(iv)$ and $(v)$ of
  Proposition~\ref{prop:Choquet:props}.
\end{proof}
Proposition~\ref{prop:riesz} then turns this $\omega$-continuous
linear function into an $\omega$-continuous valuation, defined as
follows.
\begin{definition}
  \label{defn:gmu}
  For every $\omega$-topological space $(X, \Latt)$, for every
  $g \in \Lcont (X, \Latt)$, and for every $\omega$-continuous
  valuation $\mu$ on $(X, \Latt)$, we define:
  \begin{align}
    \label{eq:density}
    (g \cdot \mu) (U) & \eqdef \int \chi_U. g \;d\mu
  \end{align}
  for every $U \in \Latt$.
\end{definition}
Lemma~\ref{lemma:gmu} and Proposition~\ref{prop:riesz} together
yield the following.
\begin{proposition}
  \label{prop:gmu}
  For every $\omega$-topological space $(X, \Latt)$, for every
  $g \in \Lcont (X, \Latt)$, and for every $\omega$-continuous
  valuation $\mu$ on $(X, \Latt)$,
  \begin{enumerate}[label=(\roman*)]
  \item $g \cdot \mu$ is an $\omega$-continuous valuation;
  \item $g \cdot \mu$ is a continuous valuation if $(X, \Latt)$ is a
    topological space and $\mu$ is a continuous valuation;
  \item For every $h \in \Lcont (X, \Latt)$,
    \begin{align}
      \label{eq:gmu:int}
      \int h \;d (g \cdot \mu) & = \int h g \;d\mu.
    \end{align}
  \end{enumerate}
\end{proposition}
In particular, if $\Latt$ is a $\sigma$-algebra, then $g \cdot \mu$ is
a measure for every measure $\mu$ and every measurable map $g$ from
$X$ to $\creal$.  The measure $g \cdot \mu$ is sometimes written as
$g \;d\mu$.

Given two valuations $\mu$ and $\nu$ on $(X, \Latt)$, one may wonder
when one can write $\nu$ as $g \cdot \mu$ for some suitable map
$g$---this is the goal of this paper.  If $\nu = g \cdot \mu$, then we
will see that $\nu$ and $\mu$ must satisfy two conditions: absolute
continuity, % (even something stronger, which we will call strong
% continuity)
and what we call the Hahn decomposition property, after the Hahn
decomposition theorem of measure theory.

\subsection{Absolute continuity} % and strong continuity}
\label{sec:absol-cont-strong}

We take the following definition of absolute continuity.  While
different from the usual definition, it is not entirely unusual, see
for example \cite{Bochner:rn}. %\cite{Fefferman:RN}.

\begin{definition}[Absolute continuity]
  \label{defn:abs:cont}
  Given two valuations $\mu$ and $\nu$ on a Pervin space $(X, \Latt)$,
  we say that $\nu$ is \emph{absolutely continuous} with respect to
  $\mu$ if and only if for every $U_0 \in \Latt$ such that
  $\nu (U_0) < \infty$, for every $\epsilon \in \Rp \diff \{0\}$,
  there is an $\eta \in \Rp \diff \{0\}$ such that for every
  $U \in \Latt$ such that $U \subseteq U_0$ and $\mu (U) < \eta$,
  $\nu (U) < \epsilon$.
\end{definition}

\begin{remark}
  \label{rem:abscont:bounded}
  When $\nu$ is a bounded valuation, the definition of absolute
  continuity simplifies to: for every $\epsilon \in \Rp \diff \{0\}$,
  there is an $\eta \in \Rp \diff \{0\}$ such that for every
  $U \in \Latt$ such that $\mu (U) < \eta$, $\nu (U) < \epsilon$.
  % Similarly, strong continuity simplifies to: for every
  % $\epsilon \in \Rp \diff \{0\}$, there is an
  % $\eta \in \Rp \diff \{0\}$ such that for every
  % $C \in \mathcal A (\Latt)$ such that $\mu (C) < \eta$,
  % $\nu (C) < \epsilon$.
\end{remark}

The usual definition of absolute continuity is given as item~(2) in
the following proposition, where we show that it is equivalent in the
case of $\sigma$-finite measures.  A valuation $\nu$ on a Pervin space
$(X, \Latt)$ is \emph{$\sigma$-finite} if and only if there is a
countable family of sets $E_n \in \Latt$, $n \in \nat$, such that
$\bigcup_{n \in \nat} E_n=X$ and $\nu (E_n) < \infty$ for each
$n \in \nat$.  Replacing $E_n$ by $\bigcup_{k=0}^n E_k$ if necessary,
we may assume that ${(E_n)}_{n \in \nat}$ is a monotone sequence.
This definition applies to measures as well, in which case we retrieve
the usual notion of $\sigma$-finiteness.  Considering
Remark~\ref{rem:abscont:bounded}, the following is well-known for
bounded measures \cite[page~422]{Billingsley:probmes}, and the proof
is entirely similar.

\begin{proposition}[Absolute continuity, simplified] %[Strong and absolute continuity, simplified]
  \label{prop:abscont:simpl}
  Let $\nu$, $\mu$ be two measures on a measurable space $(X, \Latt)$,
  and consider the following statements.
  \begin{enumerate}
    % \item $\nu$ is strongly continuous with respect to $\mu$;
  \item $\nu$ is absolutely continuous with respect to $\mu$;
  \item for every $U \in \Latt$ such that $\mu (U)=0$, $\nu (U)=0$.
  \end{enumerate}
  Then $(2)$ implies $(1)$, and $(1)$ and $(2)$ are equivalent if
  $\nu$ is $\sigma$-finite.
\end{proposition}
\begin{proof}
  Let us assume that $(2)$ holds, but not $(1)$.  There is an
  $\epsilon > 0$ and a set $U_0 \in \Latt$ such that
  $\nu (U_0) < \infty$ and, for every $n \in \nat$, letting
  $\eta \eqdef 1/2^n$, there is an element $V_n \in \Latt$ with
  $V_n \subseteq U_0$ such that $\mu (V_n) < 1/2^n$ but
  $\nu (V_n) \geq \epsilon$.  In particular,
  $\sum_{n=0}^\infty \mu (V_n) < \infty$, so by the first
  Borel-Cantelli lemma \cite[Theorem~4.3]{Billingsley:probmes},
  $\mu (\bigcap_{m \in \nat} \bigcup_{n \in \nat} V_n) = 0$.  Using
  $(2)$, it follows that
  $\nu (\bigcap_{m \in \nat} \bigcup_{n \geq m} V_n) = 0$.  The sets
  $\bigcup_{n \geq m} V_n$ form a decreasing sequence of elements of
  $\Latt$ included in $U_0$, hence of finite $\nu$-measure.  Therefore
  $\inf_{m \in \nat} \nu (\bigcup_{n \geq m} V_n) = 0$.  This is
  impossible, since for every $m \in \nat$,
  $\nu (\bigcup_{n \geq m} V_n) \geq \nu (V_m) \geq \epsilon$.

  Conversely, we assume that $(1)$ holds and that $\nu$ is
  $\sigma$-finite.  Let ${(E_n)}_{n \in \nat}$ be a monotone sequence
  of elements of $\Latt$ covering $X$ and such that
  $\nu (E_n) < \infty$ for every $n \in \nat$.  Let also $U \in \Latt$
  be such that $\mu (U)=0$.  For every $n \in \nat$, $U \cap E_n$ is
  included in $U$, and $\mu (U \cap E_n) = 0$, so by absolute
  continuity, for every $\epsilon > 0$, $\nu (U \cap E_n) < \epsilon$.
  Since $\epsilon$ is arbitrary, $\nu (U \cap E_n)=0$.  Then
  $\nu (U) = \nu (U \cap \dcup_{n \in \nat} E_n) = \dsup_{n \in \nat}
  \nu (U \cap E_n)=0$.
\end{proof}

We will use the following often.
\begin{lemma}
  \label{lemma:gmu:int}
  Let $(X, \Latt)$ be a Pervin space, $\mu$ be a
  valuation on $(X, \Latt)$ and $g \in \Lcont (X, \Latt)$.
  % \begin{enumerate}
  % \item
  For every $U \in \Latt$, $(g \cdot \mu) (U) = \int_0^\infty
  \mu (U \cap g^{-1} (]t, \infty])) \;dt$.
  % \item for every $C \in \mathcal A (\Latt)$ such that
  %   $C \subseteq U_0$ for some $U_0 \in \Latt$ satisfying
  %   % $\mu (U_0) < \infty$ and
  %   $(g \cdot \mu) (U_0) < \infty$ (e.g., if $g \cdot \mu$ is itself
  %   bounded), then
  %   $(g \cdot \mu) (C) = \int_0^\infty \mu (C \cap g^{-1} (]t,
  %   \infty])) \;dt$.
  % \end{enumerate}
\end{lemma}
% In (2), $g \cdot \mu$ should be read as any extension of $g \cdot \mu$
% to $(X, \mathcal A (\Latt))$, as afforded to us by the
% Smiley-Horn-Tarski theorem.

\begin{proof}
  Let $\nu \eqdef g \cdot \mu$.  For every $U \in \Latt$, we write
  $\nu (U) \eqdef \int \chi_U g \;d\mu$ as
  $\int_0^\infty \mu ((\chi_U g)^{-1} (]t, \infty])) \;dt$.  For every
  $t \in \Rp$,
  $(\chi_U g)^{-1} (]t, \infty]) = U \cap g^{-1} (]t, \infty])$,
  whence the result.
%
  % let $C \in \mathcal A (\Latt)$ be such that
  % $C \subseteq U_0$ for some $U_0 \in \Latt$ satisfying
  % $\nu (U_0) < \infty$.  We first deal with the case where $C$ is a
  % crescent $U \diff V$, $U, V \in \Latt$.  Replacing $U$ by
  % $U \cap U_0$ if necessary, we may assume that $U$ is included in
  % $U_0$.
%
  % We must have $\nu (C) + \nu (U \cap V) = \nu (U)$, and
  % $\nu (U \cap V), \nu (U) \leq \nu (U_0) < \infty$; hence
  % $\nu (C) = \nu (U) - \nu (U \cap V)$.  Then, by item~(1),
  % $\nu (U) = \int_0^\infty \mu (U \cap g^{-1} (]t, \infty])) \;dt$ and
  % $\nu (U \cap V) = \int_0^\infty \mu (U \cap V \cap g^{-1} (]t,
  % \infty])) \;dt$, and those are finite values.
  % % Hence it makes sense to
  % % write
  % % $\nu (C) = \int_0^\infty \mu (U \cap g^{-1} (]t, \infty])) \;dt -
  % % \int_0^\infty \mu (U \cap V \cap g^{-1} (]t, \infty])) \;dt$.
  % For every $t \in \Rp$,
  % $\mu (C \cap g^{-1} (]t, \infty])) + \mu (U \cap V \cap g^{-1} (]t,
  % \infty])) = \mu (U \cap g^{-1} (]t, \infty]))$.  By the additivity
  % of Riemann integrals, we obtain that
  % $\int_0^\infty \mu (C \cap g^{-1} (]t, \infty])) \;dt + \nu (U \cap V)
  % = \nu (U)$, so
  % $\nu (C) = \int_0^\infty \mu (C \cap g^{-1} (]t, \infty])) \;dt$.
%
  % When $C$ is a disjoint union of finitely many crescents $C_1$,
  % \ldots, $C_n$ included in $U_0$, finally,
  % $\nu (C) = \sum_{i=1}^n \nu (C_i) = \sum_{i=1}^n \int_0^\infty \mu
  % (C_i \cap g^{-1} (]t, \infty])) \;dt = \int_0^\infty \mu (C \cap
  % g^{-1} (]t, \infty])) \;dt$, again by additivity of Riemann integrals.
\end{proof}

\begin{proposition}
  \label{prop:abscont}
  Let $\mu$ and $\nu$ be two valuations on a Pervin space
  $(X, \Latt)$.  If $\nu = g \cdot \mu$ for some function
  $g \in \Lcont (X, \Latt)$,
  % , or that $\mu$ is bounded and
  % $\nu = \varphi \cdot \mu$ for some map $\varphi \eqdef f-g$, where
  % $f \geq g$ are bounded maps in $\Lcont (X, \Latt)$.
  then $\nu$ is
  % strongly continuous, and in particular
  absolutely continuous with respect to $\mu$.
\end{proposition}
\begin{proof}
  % We first deal with the case where $\nu = g \cdot \mu$. 
  Let us fix $\epsilon \in \Rp \diff \{0\}$ and $U_0 \in \Latt$ such
  that $\nu (U_0) < \infty$.
  % We also extend $\mu$ to a valuation on
  % $\mathcal A (\Latt)$, using the Smiley-Horn-Tarski theorem.  The
  % definition of strongly continuous requires us to find some extension
  % of $\mu$, but we will show that any extension fits.

  Let $h (t) \eqdef \mu (U_0 \cap g^{-1} (]t, \infty]))$, and
  $h_N (t)$ be defined as $h (t)$ if $t \leq N$, $0$ otherwise.  The
  maps $h_N$, $N \in \nat$, are antitonic, and their pointwise
  supremum is $h$.  Using Lemma~\ref{lemma:gmu:int}, with
  $U \eqdef U_0$, and Fact~\ref{fact:R:cont},
  $\nu (U_0) = \int_0^\infty h (t) \;dt = \dsup_{N \in \nat}
  \int_0^\infty h_N (t) \;dt$.  Since $\nu (U_0) < \infty$, for some
  $N \in \nat \diff \{0\}$,
  $\int_0^\infty h_N (t) \;dt > \nu (U_0) - \epsilon/2$.  Then
  % $\nu (U_0) = \int_0^\infty h (t) \;dt = \int_0^N h (t) \;dt +
  % \int_N^\infty h (t) \;dt = \int_0^N h_N (t) \;dt + \int_N^\infty h (t)
  % \;dt = \int_0^\infty h_N (t) \;dt + \int_N^\infty h (t) \;dt$, so
  $\int_N^\infty h (t) \;dt < \epsilon/2$.

  % Since $\nu = g \cdot \mu$, for every $U \in \Latt$,
  % $\nu_{|U_0} (U) = \nu (U \cap U_0) = \int_0^\infty \mu (U \cap U_0
  % \cap g^{-1} (]t, \infty])) \;dt$ (by Lemma~\ref{lemma:gmu:int})
  % $= \int_0^\infty \mu_{|U_0} (U \cap g^{-1} (]t, \infty])) \;dt$, so
  % $\nu_{|U_0} = g \cdot \mu_{|U_0}$.

  Let $\eta \eqdef \epsilon / (2N)$.  For every $U \in \Latt$ such
  that $U \subseteq U_0$ and $\mu (U) < \eta$, % (i.e.,
  % $\mu_{|U_0} (U) < \eta$)
  we show that % $\nu_{|U_0} (U) < \epsilon$,
  % namely 
  $\nu (U) < \epsilon$ as follows.% We note that
  % $\nu (U_0) < \infty$, so that we can apply
  % Lemma~\ref{lemma:gmu:int}.
  \begin{align*}
    \nu (U)
    & = (g \cdot \mu) (U)
    \\
    & = \int_0^\infty \mu (U \cap g^{-1} (]t, \infty])) \;dt
    \quad\text{by Lemma~\ref{lemma:gmu:int}} \\
    & = \int_0^N \mu (U \cap g^{-1} (]t, \infty])) \;dt + \int_N^\infty \mu
      (U \cap g^{-1} (]t, \infty])) \;dt \\
    & \leq N \mu (U) + \int_N^\infty h (t) \;dt
      < N \mu (U) + \epsilon/2 < N \eta + \epsilon/2 = \epsilon.  \qedhere
  \end{align*}
%
  % We now deal with the case where $\mu$ is bounded and
  % $\nu = \varphi \cdot \mu$ for some map $\varphi \eqdef f-g$, where
  % $f \geq g$ are bounded maps in $\Lcont (X, \Latt)$.  Again, we fix
  % $\epsilon \in \Rp \diff \{0\}$ and $U_0 \in \Latt$ such that
  % $\nu (U_0) < \infty$.  For every crescent $C \eqdef U \diff V$ where
  % $U, V \in \Latt$ and $U \subseteq U_0$,
  % $(\varphi \cdot \mu)_{|U_0} (C) = (\varphi \cdot \mu) (U_0 \cap U
  % \diff V) = (\varphi \cdot \mu) (U_0 \cap U) - (\varphi \cdot \mu)
  % (U_0 \cap U \cap V) = (f \cdot \mu) (U_0 \cap U) - (g \cdot \mu)
  % (U_0 \cap U) - (f \cdot \mu) (U_0 \cap U \cap V) + (g \cdot \mu)
  % (U_0 \cap U \cap V) = (f \cdot \mu) (U_0 \cap C) - (g \cdot \mu)
  % (U_0 \cap C) = (f \cdot \mu)_{|U_0} (C) - (g \cdot \mu)_{|U_0} (C)$.
  % This also holds when $C$ is a finite disjoint union of crescents, by
  % the finite additivity of valuations on $\mathcal A (\Latt)$.
  %
  % In the previous paragraph, we have found an $\eta > 0$ such that,
  % for every $C \in \mathcal A (\Latt)$ included in $U_0$ and
  % satisfying $\mu_{|U_0} (C) < \eta$,
  % $(f \cdot \mu)_{|U_0} (C) < \epsilon$.  It follows that
  % $(\varphi \cdot \mu)_{|U_0} (C) \leq (f \cdot \mu)_{|U_0} (C) <
  % \epsilon$.
\end{proof}

\subsection{Absolute continuity is not enough}
\label{sec:absol-cont-not}

Given a topological space $(X, \Latt)$, let $\Borel \Latt$ be its
Borel $\sigma$-algebra.  A \emph{Borel measure}, namely a measure on
$(X, \Borel\Latt)$, induces a valuation on $(X, \Latt)$ by restriction
to the open sets.  The Borel measures for which the induced valuation
is continuous are traditionally called \emph{$\tau$-smooth}.  By
Adamski's theorem \cite[Theorem~3.1]{Adamski:measures}, it is
equivalent to require all Borel measures on $(X, \Borel\Latt)$ to be
$\tau$-smooth, or to require $(X, \Latt)$ to be \emph{hereditarily
  Lindel\"of}; a space is hereditarily Lindel\"of if and only if every
family ${(U_i)}_{i \in I}$ of open subsets has a countable subfamily
with the same union.  All second-countable spaces are hereditarily
Lindel\"of.

There has been quite some literature on the converse question, among
which \cite{Lawson:valuation,alvarez-manilla00,KL:measureext}: given a
continuous valuation $\nu$ on $(X, \Latt)$, does $\nu$ extend to a
(necessarily $\tau$-smooth) Borel measure?  One of the most general
theorems of this kind is the following \cite[Theorem~1]{dBGLJL:LCS}:
every continuous valuation on an LCS-complete space extends to a Borel
measure; an \emph{LCS-complete} space is a homeomorph of a $G_\delta$
subset of a locally compact sober space.  The class of LCS-complete
spaces includes all locally compact sober spaces, Matthew de Brecht's
quasi-Polish spaces \cite{deBrecht:qPolish}, and therefore also all
Polish spaces.

Additionally, a standard use of the $\pi\lambda$-theorem
\cite[Theorem~3.2]{Billingsley:probmes} shows that any $\sigma$-finite
continuous valuation $\nu$ on $(X, \Latt)$ extends to a \emph{unique}
Borel measure.  That Borel measure $\mu$ is such that there exists a
monotone sequence ${(U_n)}_{n \in \nat}$ of open sets covering $X$ and
such that $\mu (U_n) < \infty$ for every $n \in \nat$.  This is a
stricter condition than simply being $\sigma$-finite, since $U_n$ is
required to be open, and Borel measures having this property are
sometimes called \emph{moderated}.

Since quasi-Polish spaces are second-countable
\cite[Definition~16]{deBrecht:qPolish}, hence hereditarily Lindel\"of,
it follows that $\sigma$-finite continuous valuations are in
one-to-one correspondence with moderated $\tau$-smooth measures on
quasi-Polish spaces.

We use this to transport the classical Radon-Nikod\'ym theorem over to
the world of continuous valuations.

In one direction, given any $\sigma$-finite continuous valuation $\mu$
on an LCS-complete space $(X, \Latt)$, let $\widetilde\mu$ be its
unique extension to a Borel measure. % (on $(X, \Borel\Latt)$).
For every measurable map (not just any lower semicontinuous map) $g$,
in $\Lcont (X, \Borel\Latt)$, we can form the measure
$g \cdot \widetilde\mu$ on $(X, \Borel\Latt)$.  This induces an
$\omega$-continuous valuation by restriction to $\Latt$, which we
write as $g \cdot \mu$, extending Definition~\ref{defn:gmu} to a
larger class of density functions.  With this definition, we have the
following.
%This measure is
% $\tau$-smooth, namely it induces a continuous valuation on
% $(X, \Latt)$.  Indeed, 

\begin{theorem}
  \label{thm:RN:val}
  For any two $\sigma$-finite continuous valuations on an LCS-complete
  space $(X, \Latt)$, the following are equivalent:
  \begin{enumerate}
  \item $\widetilde\nu$ is absolutely continuous with respect to
    $\widetilde\mu$;
  % \item for every $U_0 \in \Latt$ such that $\nu (U_0) < \infty$, for
  %   every $\epsilon \in \Rp \diff \{0\}$, there is an
  %   $\eta \in \Rp \diff \{0\}$ such that for every crescent $C$ such
  %   that $C \subseteq U_0$ and $\widetilde\mu (C) < \eta$,
  %   $\widetilde\nu (C) < \epsilon$;
  \item there is a measurable map $g \in \Lcont (X, \Borel\Latt)$ such
    that $\nu = g \cdot \mu$.
  \end{enumerate}
  Additionally, $g$ is unique up to $\widetilde\mu$-null sets.
\end{theorem}
% The difference between Definition~\ref{defn:abs:cont} and item~(2)
% above is that the former quantifies over open sets $U$, while (2)
% quantifies over crescents $C$.  A \emph{crescent} is a difference
% $U \diff V$ of elements $U, V \in \Latt$; we can assume
% $V \subseteq U$ without loss of generality.  It will turn out that
% $\widetilde\mu (C)$ and $\widetilde\nu (C)$ will soon simply be
% written as $\mu (C)$ and $\nu (C)$ respectively, being the unique
% images of $C$ by any extension of $\mu$, resp.\ $\nu$, obtained by the
% Smiley-Horn-Tarski theorem (see Section~\ref{sec:smiley-horn-tarski}).

\begin{proof}
  The condition $\nu = g \cdot \mu$ is equivalent to
  $\widetilde\nu = g \cdot \widetilde\mu$, by our (re)definition of
  $g \cdot \mu$.  We conclude by invoking the classical
  Radon-Nikod\'ym theorem.
%
  % Item~(1) implies (2) since every $U_0 \in \Latt$ is in
  % $\Borel\Latt$, and since every crescent is in $\Borel\Latt$.
%
  % Let us show that (2) implies (1).
\end{proof}

Although this is a positive result, this gives us a recipe to show
that absolute continuity is \emph{not} enough for two $\sigma$-finite
$\omega$-continuous valuation to have a density
$g \in \Lcont (X, \Latt)$: find measurable maps that are equal to no
lower semicontinuous map up to a $\widetilde\mu$-null set.

We provide two counter-examples.  The first one relies on the
existence of non-trivial specialization orderings in non-$T_1$ spaces.
The second one takes place in $\real$ with its standard metric
topology.

\begin{example}
  \label{ex:hahn:needed:mono}
  Let $\mu \eqdef a \delta_x + b \delta_y$, where $a, b > 0$ and $x$
  and $y$ are two points of an LCS-complete space $(X, \Latt)$ with
  $x < y$.  (We let $x \leq y$ if and only if every $U \in \Latt$
  containing $x$ contains $y$; this is the \emph{specialization
    preordering} of $(X, \Latt)$.  A space is $T_0$ if and only if
  $\leq$ is antisymmetric, and every LCS-complete space is $T_0$.  We
  write $x < y$ if $x \leq y$ and $y \not\leq x$.)  Next, consider any
  $g \in \Lcont (X, \Borel\Latt)$ such that $g (x) > g (y)$.  For
  example, taking $g \eqdef 1-h$ fits, where $h$ is any lower
  semicontinuous map from $(X, \Latt)$ to $[0, 1] \subseteq \creal$
  such that $h (x) \neq h (y)$; indeed, every lower semicontinuous map
  is monotonic.  We note that $g$ is equal to no lower semicontinuous
  map up to any $\widetilde\mu$-null set, because $g$ is antitonic,
  lower semicontinuous maps are monotonic, and the
  $\widetilde\mu$-null sets are the Borel sets that contain neither
  $x$ nor $y$.  Therefore $g \cdot \mu$ has no lower semicontinuous
  density with respect to $\mu$.  For a concrete instance of this
  construction, consider Sierpi\'nski space for $(X, \Latt)$, namely
  $(\{0, 1\}, \{\emptyset, \{1\}, \{0, 1\}\})$, $x \eqdef 0$,
  $y \eqdef 1$, $g (0) \eqdef 1$, $g (1) \eqdef 0$.
\end{example}

\begin{example}
  \label{ex:hahn:needed:hausdorff}
  Let $\mu$ be the bounded discrete valuation
  $\delta_0 + \sum_{n \in \nat} \frac 1 {2^n} \delta_{1/2^n}$ on
  $\real$ with its standard topology. Let $g$ map every non-zero real
  number to $0$, and $0$ to $1$.  This is a measurable map.  The
  $\widetilde\mu$-null sets are the Borel sets that do not contain $0$
  or any point $1/2^n$, $n \in \nat$.  If $g$ were equal to some
  $h \in \Lcont (\real)$ up to some $\widetilde\mu$-null set, then we
  would have $h (0)=1$ and $h (1/2^n)=0$ for every $n \in \nat$.  But
  then $h^{-1} (]1/2, \infty])$ would contain $0$, hence $1/2^n$ for
  $n$ large enough, and that is impossible since $h (1/2^n)=0$.  It
  follows that $g \cdot \mu$ has no lower semicontinuous density with
  respect to $\mu$.
\end{example}

We will therefore look for additional conditions imposed by the
existence of $g \in \Lcont (X, \Latt)$ such that $\nu = g \cdot \mu$.

\subsection{The Smiley-Horn-Tarski theorem}
\label{sec:smiley-horn-tarski}

Let $\mathcal A (\Latt)$ be the smallest algebra of subsets of $X$
containing $\Latt$.  Its elements are the unions of finite collections
of pairwise disjoint crescents.  A \emph{crescent} is a difference
$U \diff V$ of elements $U, V \in \Latt$; we can assume
$V \subseteq U$ without loss of generality.

The \emph{Smiley-Horn-Tarski theorem}
\cite{smiley44,HornTarski48:ext,Pettis:ext} states that every bounded
valuation $\nu$ on $(X, \Latt)$ extends to a unique (bounded)
valuation on $(X, \mathcal A (\Latt))$.  In general, every valuation
$\nu$ on $(X, \Latt)$ (not necessarily bounded) extends to a valuation
on $(X, \mathcal A (\Latt))$, but that extension may fail to be unique
\cite[Proposition~IV-9.4]{GHKLMS:contlatt}.  We will usually write an
extension of $\nu$ on $(X, \mathcal A (\Latt))$ with the same letter
$\nu$, although one should be careful that such extensions may fail to
be unique when $\nu$ is not bounded.  Still, some uniqueness remains:
if $C \in \mathcal A (\Latt)$ can be written as a disjoint union of
crescents $U_i \diff V_i$ with $V_i \subseteq U_i$ ($1\leq i \leq n$),
and if $\nu (U_i) < \infty$ for every $i$, then necessarily
$\nu (C) = \sum_{i=1}^n (\nu (U_i) - \nu (V_i))$.

\begin{lemma}
  \label{lemma:gmu:int2}
  Let $(X, \Latt)$ be a Pervin space, $\mu$ be a bounded valuation on
  $(X, \Latt)$ and $g \in \Lcont (X, \Latt)$.  The function $\nu$ that
  maps every $C \in \mathcal A (\Latt)$ to
  $\int_0^\infty \mu (C \cap g^{-1} (]t, \infty])) \;dt$ is a
  valuation on $\mathcal A (\Latt)$ that extends $g \cdot \mu$ to
  $(X, \mathcal A (\Latt))$.
\end{lemma}
We call the valuation $\nu$ above the \emph{canonical extension} of
$g \cdot \mu$ to $(X, \mathcal A (\Latt))$.  There may be others:
while $\mu$ is bounded, $g \cdot \mu$ may fail to be.
\begin{proof}
  The definition of $\nu (C)$ makes sense, since the extension of
  $\mu$ to $\mathcal A (\Latt)$, which is required to make sense of
  $\mu (C \cap g^{-1} (]t, \infty]))$ is unique, owing to the fact
  that $\mu$ is bounded.  It is clear that $\nu (\emptyset)=0$.  The
  modularity and the monotonicity of $\nu$ on $\mathcal A (\Latt)$
  follow from the modularity and the monotonicity of $\mu$.  Hence
  $\nu$ is a valuation, and it extends $g \cdot \mu$ by
  Lemma~\ref{lemma:gmu:int}.
\end{proof}

Since extensions to $\mathcal A (\Latt)$ are unique for bounded
valuations, we obtain the following.
\begin{corollary}
  \label{corl:gmu:int2}
  Let $(X, \Latt)$ be a Pervin space, $\mu$ be a bounded valuation on
  $(X, \Latt)$ and $g \in \Lcont (X, \Latt)$.  If $\nu \eqdef g \cdot
  \mu$ is bounded, then its unique extension to $\mathcal A (\Latt)$
  is such that $\nu (C) = \int_0^\infty \mu (C \cap g^{-1} (]t,
  \infty])) \;dt$ for every $C \in \mathcal A (\Latt)$.
\end{corollary}

\subsection{Signed valuations}
\label{sec:signed-valuations}

In order to state the Hahn decomposition property, we need to
introduce signed valuations.

A \emph{signed valuation} is a map $\varsigma \colon \Latt \to \real$
(not $\Rp$) that is \emph{strict} ($\varsigma (\emptyset) = 0$) and
\emph{modular} (for all $U, V \in \Latt$,
$\varsigma (U \cup V) + \varsigma (U \cap V) = \varsigma (U) +
\varsigma (V)$).

Typical examples of signed valuations are given by maps
$\nu - r \cdot \mu \colon \Latt \to \real$, where $\nu$ and $\mu$ are
bounded valuations on $(X, \Latt)$, and every $r \in \Rp$.

We have the following analogue of the bounded form of the
Smiley-Horn-Tarski theorem.  The proof uses ingredients similar to
Proposition~\ref{prop:riesz}, and can also be used to derive the
classical Smiley-Horn-Tarski theorem.
\begin{proposition}
  \label{prop:pettis:signed}
  Let $\Latt$ be a lattice of subsets of a set $X$, and $\varsigma$ be
  a signed valuation on $(X, \Latt)$.  Then $\varsigma$ extends to a
  unique signed valuation on $(X, \mathcal A (\Latt))$.  The
  extension---still written $\varsigma$---satisfies
  $\varsigma (U \diff V) = \varsigma (U) - \varsigma (U \cap V) =
  \varsigma (U \cup V) - \varsigma (V)$ for all $U, V \in \Latt$.
\end{proposition}
\begin{proof}
  If $\varsigma^\%$ is any signed valuation extending $\varsigma$ on
  $(X, \mathcal A (\Latt))$, then it is defined uniquely on crescents
  by the fact that $\varsigma^\% (U \diff V)$ must be equal to
  $\varsigma (U) - \varsigma (U \cap V)$ and also to
  $\varsigma (U \cup V) - \varsigma (V)$ for all $U, V \in \Latt$, by
  modularity and the fact that
  $\varsigma^\% ((U \cap V) \cap (U \diff V))$ and
  $\varsigma^\% (V \cap (U \diff V))$ must both be equal to
  $\varsigma^\% (\emptyset)=0$; then $\varsigma^\%$ is uniquely
  determined on finite disjoint unions of crescents by additivity.

  We proceed as follows to prove the existence of $\varsigma^\%$.  Let
  $M^+$ be the set of functions $h \in \Lcont (X, \Latt)$ taking their
  values in $\nat$.  One can write any such $h$ as
  $\sum_{i=1}^\infty \chi_{U_i}$ in a unique way, where
  $U_1 \supseteq \cdots \supseteq U_n \supseteq \cdots$ form an
  antitone sequence of elements of $\Latt$, with $U_n = \emptyset$ for
  $n$ large enough.  Indeed, $U_i$ is determined uniquely as
  $h^{-1} ([i, \infty])$ (which is equal to
  $h^{-1} (]i-\epsilon, \infty])$ for any $\epsilon \in {]0,1[}$,
  hence is in $\Latt$).  On every such $h \in M^+$, let
  $F (h) \eqdef \sum_{i=1}^\infty \varsigma (U_i)$.  This is a finite
  sum, because $\varsigma$ is strict.

  With $h$ as above and $U \in \Latt$, $h+\chi_U$ is equal to
  $\sum_{i=1}^\infty \chi_{V_i}$ where $V_i = {(h+ \chi_U)}^{-1} ([i,
  \infty]) = h^{-1} ([i, \infty]) \cup (h^{-1} ([i-1, \infty]) \cap
  U) = U_i \cup (U_{i-1} \cap U)$; when $i=1$, we use the convention
  that $U_0=X$.  Hence:
  \begin{align*}
    F (h+\chi_U) & = \sum_{i=1}^\infty \varsigma (U_i \cup (U_{i-1}
                   \cap U)) \\
                 & = \sum_{i=1}^\infty \left(\varsigma (U_i) + \varsigma (U_{i-1} \cap U)
                   - \varsigma (U_i \cap U)\right) \\
                 & \qquad\text{by modularity; note that $U_i \cap
                   U_{i-1} \cap U$ simplifies to $U_i \cap U$}
                 % since $U_i \subseteq U_{i-1}$}
    \\
                 & = F (h) + \varsigma (U),
  \end{align*}
  by canceling the telescoping terms $\varsigma (U_{i-1} \cap U)$ and
  $\varsigma (U_i \cap U)$, so that only
  $\varsigma (U_0 \cap U) = \varsigma (U)$ remains.
  
  For every $h' \in M^+$, written as $\sum_{j=1}^\infty \chi_{V_j}$
  where $V_1 \supseteq \cdots \supseteq V_n \supseteq \cdots$ form an
  antitone sequence of elements of $\Latt$, with $V_n = \emptyset$ for
  $n$ large enough, we obtain that $F (h+h') = F (h) + F (h')$ by
  induction on the number of non-empty sets $V_n$.

  We can therefore extend $F$ to an additive map from $M$ to $\Z$,
  where $M$ is the collection of differences $f-g$ of two elements of
  $M^+$, by $F (f-g) \eqdef F (f) - F (g)$.  This is unambiguous: if
  $f-g=f'-g'$, then $f+g' = f'+g$, so $F (f)+F(g')=F(f')+F(g)$, or
  equivalently $F (f)-F(g) = F(f')-F(g')$.

  Let us define $\varsigma^+ (C) \eqdef F (\chi_C)$ for every subset
  $C$ of $X$ such that $\chi_C \in M$.  Amongst those, we find the
  crescents $U \diff V$ (with $U, V \in \Latt$ and $V \subseteq U$),
  since $\chi_{U \diff V} = \chi_U - \chi_V$.  We also find the finite
  disjoint unions of crescents $C_1$, \ldots, $C_n$, since their
  characteristic map is $\sum_{i=1}^n \chi_{C_i}$.  Now $\varsigma^+$
  is strict since $F (0)=0$, and modular on $(X, \mathcal A (\Latt))$.
  The latter rests on the fact that for any sets $C$ and $C'$,
  $\chi_{C \cup C'} + \chi_{C \cap C'} = \chi_C + \chi_{C'}$: then
  $\varsigma (C \cup C') + \varsigma (C \cap C') = F (\chi_{C \cup C'}
  + \chi_{C \cap C'})$ (since $F$ is additive)
  $= F (\chi_C + \chi_{C'}) = \varsigma (C) + \varsigma (C')$ (since
  $F$ is additive, once again).
\end{proof}

\subsection{The Hahn decomposition property}
\label{sec:hahn-decomp-prop}

\begin{definition}[Hahn decomposition property]
  \label{defn:hahn}
  Let $(X, \Latt)$ be a Pervin space.  A signed valuation
  $\varsigma \colon \Latt \to \real$ has the \emph{Hahn decomposition
    property} if and only if there is an element $U$ of $\Latt$ such
  that:
  \begin{itemize}
  \item for every crescent $C$ included in $U$, $\varsigma (C) \geq 0$;
  \item for every crescent $C$ disjoint from $U$, $\varsigma (C) \leq 0$.
  \end{itemize}
\end{definition}
In this definition, we extend $\varsigma$ implicitly to a valuation on
$(X, \mathcal A (\Latt))$, using Proposition~\ref{prop:pettis:signed},
in order to make sense of $\varsigma (C)$.  We will call the set $U$
given above a \emph{witness} to the Hahn decomposition property.

\begin{proposition}
  \label{prop:hahn}
  Let $\mu$ and $\nu$ be two bounded valuations on a Pervin space
  $(X, \Latt)$.  If $\nu = g \cdot \mu$ for some
  $g \in \Lcont (X, \Latt)$, then for every $r \in \Rp$, the signed
  valuation $\nu - r \cdot \mu$ has the Hahn decomposition
  property---and one can take $U \eqdef g^{-1} (]r, \infty])$ as a
  witness to the latter.
\end{proposition}
\begin{proof}
  We take $U \eqdef g^{-1} (]r, \infty])$.  For every crescent
  $C \subseteq U$, $C \cap g^{-1} (]t, \infty]) = C$ for every
  $t \in [0, r]$, since $g^{-1} (]t, \infty])$ contains $U$ in that
  case.  Hence:
  \begin{align*}
    \nu (C) & = \int_0^\infty \mu (C \cap g^{-1} (]t, \infty])) \;dt
            \qquad\qquad \text{by Corollary~\ref{corl:gmu:int2}} \\
            & = \int_0^r  \mu (C \cap g^{-1} (]t, \infty])) \;dt + \int_r^\infty
              \mu (C \cap g^{-1} (]t, \infty])) \;dt \\
            & \geq \int_0^r  \mu (C \cap g^{-1} (]t, \infty])) \;dt
              = \int_0^r \mu (C) \;dt = r \cdot \mu (C).
  \end{align*}
  For every crescent $C$ disjoint from $U$,
  $C \cap g^{-1} (]t, \infty])$ is empty for every $t \geq r$, since
  $g^{-1} (]t, \infty])$ is included in $U$ in that case.  Hence:
  \begin{align*}
    \nu (C) & = \int_0^\infty \mu (C \cap g^{-1} (]t, \infty])) \;dt
              %& \text{by Lemma~\ref{lemma:gmu:int2}}
    \\
    % & = \int_0^r  \mu (C \cap g^{-1} (]t, \infty])) \;dt + \int_r^\infty
    %   \mu (C \cap g^{-1} (]t, \infty])) \;dt \\
              % & \qquad\qquad \text{by the Chasles relation
              %   (Theorem~\ref{thm:Riemann},
              %   item~\ref{thm:Riemann:Chasles})} \\
    & = \int_0^r  \mu (C \cap g^{-1} (]t, \infty])) \;dt
      \leq \int_0^r \mu (C) \;dt = r \cdot \mu (C).
  \end{align*}
\end{proof}

Given any valuation $\nu$ on a Pervin space $(X, \Latt)$, and any
$U_0 \in \Latt$, we can define a valuation $\nu_{|U_0}$ by letting
$\nu_{|U_0} (U) \eqdef \nu (U \cap U_0)$ for every $U \in \Latt$; then
$\nu_{|U_0}$ is an $\omega$-continuous (resp., continuous) valuation
if $\nu$ is.  We also note that $\nu_{|U_0}$ is a bounded valuation if
and only if $\nu (U_0) < \infty$.  We use this in the proof of the
following corollary, and we will use the notion again later.
\begin{corollary}
  \label{corl:hahn}
  Let $\mu$ and $\nu$ be two valuations on a Pervin space
  $(X, \Latt)$.  If $\nu = g \cdot \mu$ for some
  $g \in \Lcont (X, \Latt)$, then for every $U_0 \in \Latt$ such that
  $\nu (U_0) < \infty$ and $\mu (U_0) < \infty$, for every
  $r \in \Rp$, the signed valuation $\nu_{|U_0} - r \cdot \mu_{|U_0}$
  has the Hahn decomposition property.
\end{corollary}
\begin{proof}
  If $\nu (U_0) < \infty$ and $\mu (U_0) < \infty$, then
  $\nu_{|U_0} = g \cdot \mu_{|U_0}$, since for every $U \in \Latt$,
  $\nu_{|U_0} (U) = \nu (U \cap U_0) = \int_0^\infty \mu (U \cap U_0
  \cap g^{-1} (]t, \infty])) \;dt$ (by Lemma~\ref{lemma:gmu:int})
  $= \int_0^\infty \mu_{|U_0} (U \cap g^{-1} (]t, \infty])) \;dt = (g
  \cdot \mu_{|U_0}) (U)$.  We conclude by using
  Proposition~\ref{prop:hahn}.
\end{proof}

\section{The existence of density maps}
\label{sec:exist-dens-maps}

We now show that absolute continuity and the Hahn decomposition
property suffice to guarantee the existence of a density function.
The following are the two key lemmata.  We write $\D$ for the set of
dyadic numbers, namely rational numbers of the form $p/2^n$ with
$p \in \Z$ and $n \in \nat$.  We also use the Smiley-Horn-Tarski
theorem in order to make sense of $\nu (C)$ below, and the canonical
extension given in Lemma~\ref{lemma:gmu:int2} in order to make sense
of $(g \cdot \mu) (C)$.
\begin{lemma}
  \label{lemma:gmu:below}
  Let $(X, \Latt)$ be a Pervin space, $g \in \Lcont (X, \Latt)$, and
  $\nu$, $\mu$ be two bounded valuations on $(X, \Latt)$.  Let us
  assume that for every non-negative dyadic number
  $r \in \D \cap \Rp$, for every crescent
  $C \subseteq g^{-1} (]r, \infty])$, $\nu (C) \geq r \cdot \mu (C)$.
  Then for every $C \in \mathcal A (\Latt)$,
  $\nu (C) \geq (g \cdot \mu) (C)$. In particular,
  $\nu \geq g \cdot \mu$ on $(X, \Latt)$.
\end{lemma}
\begin{proof}
    \begin{figure}
    \centering
    \includegraphics[scale=0.45]{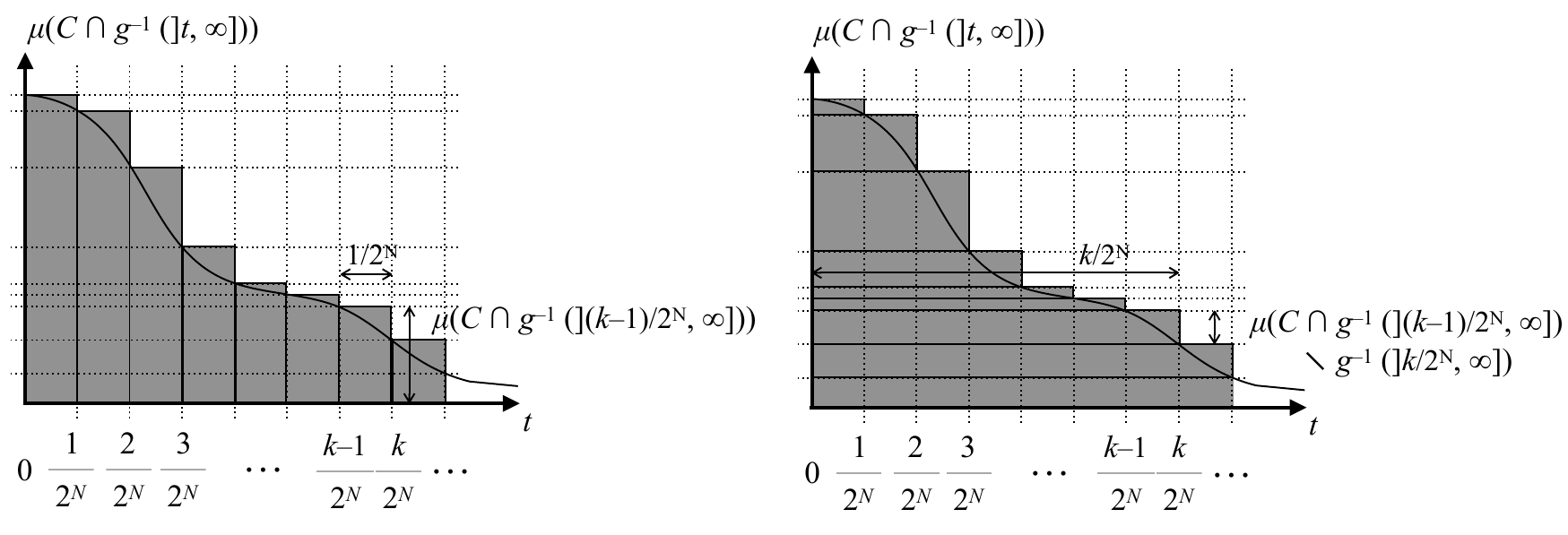}
    \caption{Bounding $(g \cdot \mu) (C)$ from above}
    \label{fig:radon-nikodym-upper}
  \end{figure}
  It suffices to show the claim for every crescent $C$.  Once this is
  done, the claim that $\nu (C) \geq (g \cdot \mu) (C)$ for every $C
  \in \mathcal A (\Latt)$ follows from the fact that $C$ is a disjoint union
  of crescents, and that $\nu$ and $g \cdot \mu$ are additive.

  We fix a crescent $C$. By definition of canonical extensions
  (Lemma~\ref{lemma:gmu:int2}),
  $(g \cdot \mu) (C) = \int_0^\infty \mu (C \cap g^{-1} (]t, \infty]))
  \;dt$.

  The main ingredient of the proof is summarized in
  Figure~\ref{fig:radon-nikodym-upper}: the sum of the areas of the
  vertical bands on the left is equal to the sum of the areas of the
  horizontal bands on the right.  We will rely on that figure in what
  follows.

  Let $f (t) \eqdef \mu (C \cap g^{-1} (]t, \infty]))$, and
  $f_N (t) \eqdef f (t)$ if $t \leq N$, $0$ otherwise.  In the figure,
  $f$ is shown as the solid decreasing curve, both on the left-hand
  side and on the right-hand side.  Since $f$ is the pointwise
  supremum of ${(f_N)}_{N \in \nat}$,
  $(g \cdot \mu) (C) = \dsup_{N \in \nat} \int_0^\infty f_N (t) \;dt$
  by Fact~\ref{fact:R:cont}.

  We fix an arbitrary $r \in \Rp$ such that $r < (g \cdot \mu) (C)$.
  For $N \in \nat$ large enough,
  $r \leq \int_0^\infty f_N (t) \;dt = \int_0^N \mu (C \cap g^{-1}
  (]t, \infty])) \;dt \leq \sum_{k=1}^{N2^N} \frac 1 {2^N} \mu (C \cap
  g^{-1} (](k-1)/2^N, \infty])$.  The latter is the sum of the areas
  of the vertical bands on the left of
  Figure~\ref{fig:radon-nikodym-upper}.
  
  Reorganizing the summation, that is also equal to the sum of the
  areas of the horizontal bands on the right, so:
  \begin{align*}
    r & \leq \sum_{k=1}^{N2^N} \frac k {2^N} \mu (C \cap g^{-1} (](k-1)/2^N,
        \infty]) \diff g^{-1} (]k/2^N, \infty]))\\
      & \qquad\qquad + N \mu (C \cap g^{-1} (]N, \infty])).
  \end{align*}
  The final term in the sum is the area of the bottommost band.  The
  sum of the terms with $1\leq k\leq N$ is bounded from above by
  $\sum_{k=1}^N \frac k {2^N} \mu (C) \leq \frac {N(N+1)} {2^{N+1}}
  \mu (C)$.  For every $k$ between $N+1$ and $N2^N$, the crescent
  $C' \eqdef C \cap g^{-1} (](k-1)/2^N, \infty]) \diff g^{-1} (]k/2^N,
  \infty])$ is included in $g^{-1} (](k-1)/2^N, \infty])$, so
  $\nu (C') \geq (k-1)/2^N \;\mu (C')$ by assumption.  Similarly,
  $\nu (C \cap g^{-1} (]N, \infty])) \geq N \; \mu (C \cap g^{-1} (]N,
  \infty]))$.

  It follows that:
  \begin{align*}
    r & \leq \frac {N(N+1)} {2^{N+1}} \mu (C) + \sum_{k=N+1}^{N2^N} \frac
        k {k-1} \nu (C \cap g^{-1} (](k-1)/2^N, \infty]) \diff g^{-1}
        (]k/2^N, \infty])) \\
    & \qquad\qquad + \nu (C \cap g^{-1} (]N, \infty])).
  \end{align*}
  In the
  middle sum, $k/(k-1)$ is smaller than or equal to $(N+1)/N$.  We
  also have
  $\nu (C \cap g^{-1} (]N, \infty])) \leq \frac {N+1} N \nu (C \cap
  g^{-1} (]N, \infty]))$, because $\frac {N+1} N \geq 1$.  Hence:
  \begin{align*}
    r & \leq \frac {N(N+1)} {2^{N+1}} \mu (C) + \frac {N+1} N
        \sum_{k=N+1}^{N2^N} \nu (C \cap g^{-1} (](k-1)/2^N, \infty]) \diff
        g^{-1} (]k/2^N, \infty])) \\
    & \qquad\qquad + \frac {N+1} N \nu (C \cap g^{-1} (]N, \infty]))
  \end{align*}
  By the additivity of $\nu$, the right-hand side is equal to
  $\frac {N(N+1)} {2^{N+1}} \mu (C) + \frac {N+1} N \nu (C \cap g^{-1}
  (]N/2^N, \infty]))$.  Since $C \cap g^{-1} (]N/2^N, \infty])$ is
  included in $C$, and $\nu$ is monotonic,
  $r \leq \frac {N(N+1)} {2^{N+1}} \mu (C) + \frac {N+1} N \nu (C)$.
  We let $N$ tend to $\infty$, and we obtain that $r \leq \nu (C)$.
  Taking suprema over all $r < (g \cdot \mu) (C)$,
  $(g \cdot \mu) (C) \leq \nu (C)$.
\end{proof}

We have a somewhat symmetric situation in the following lemma, except
that we cannot conclude that $(g \cdot \mu) (C) \geq \nu (C)$ without
further assumptions.  Once again, we use the canonical extension of $g
\cdot \mu$ to make sense of $(g \cdot \mu) (C)$.
\begin{lemma}
  \label{lemma:gmu:above}
  Let $(X, \Latt)$ be a Pervin space, $g \in \Lcont (X, \Latt)$, and
  $\nu$, $\mu$ be two bounded valuations on $(X, \Latt)$.  Let us
  assume that for every non-negative dyadic number
  $r \in \D \cap \Rp$, for every crescent $C$ disjoint from
  $g^{-1} (]r, \infty])$, $\nu (C) \leq r \cdot \mu (C)$.  Then there
  is an directed countable family ${(U_N)}_{N \in \nat}$ of elements
  of $\Latt$ with the following properties:
  \begin{enumerate}[label=(\roman*)]
  \item $\nu (X \diff \dcup_{N \in \nat} U_N)=0$;
  \item for every $C \in \mathcal A (\Latt)$, for every $N \in \nat$,
    $(g \cdot \mu) (C) \geq \frac N {N+1} \nu (C \cap U_N) + N (\mu (C
    \cap U_N) - \frac 1 {N+1} \nu (C \cap U_N))$;
  \end{enumerate}
\end{lemma}
\begin{proof}
  \begin{figure}
    \centering
    \includegraphics[scale=0.45]{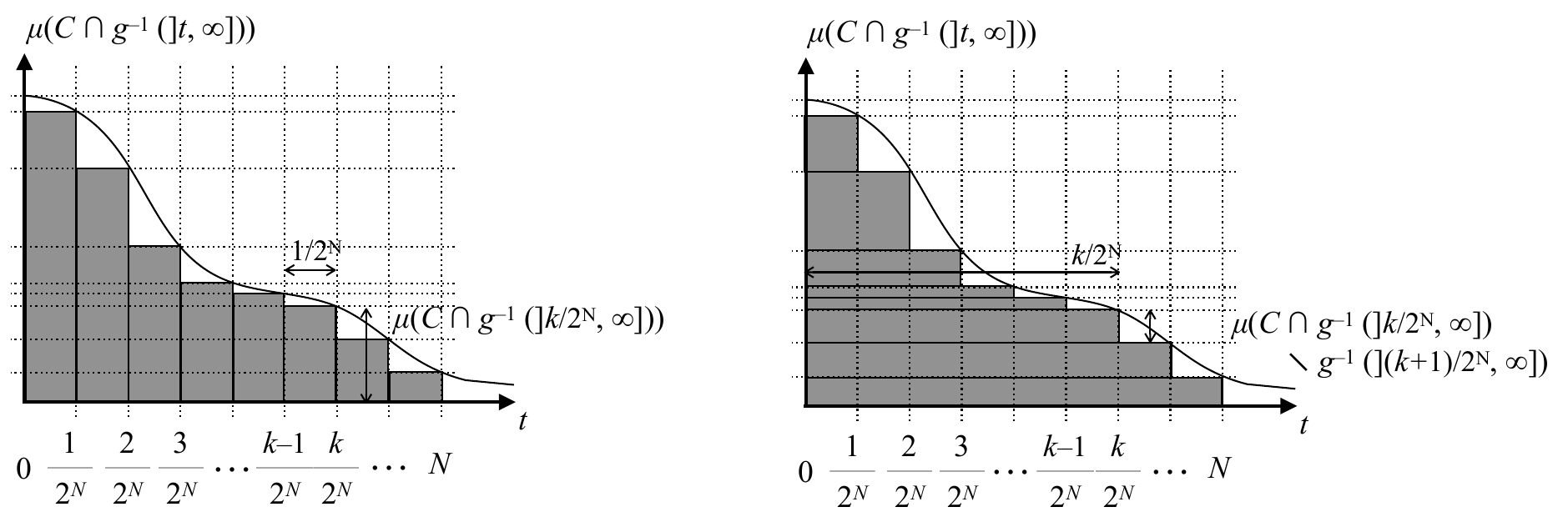}
    \caption{Bounding $(g \cdot \mu) (C)$ from below}
    \label{fig:radon-nikodym-lower}
  \end{figure}
  By definition of canonical extensions (Lemma~\ref{lemma:gmu:int2}),
  $(g \cdot \mu) (C) = \int_0^\infty \mu (C \cap g^{-1} (]t, \infty]))
  \;dt \geq \sum_{k=1}^{N2^N} \frac 1 {2^N} \mu (C \cap g^{-1}
  (]k/2^N, \infty])) \;dt$.  The latter is the area of the vertical
  bands on the left of Figure~\ref{fig:radon-nikodym-lower}, which
  rewrites as the area of the horizontal bands on the right, namely:
  \begin{align*}
    & \sum_{k=1}^{N2^N-1} \frac k {2^N} \mu (C \cap g^{-1} (]k/2^N,
      \infty]) \diff g^{-1} (](k+1)/2^N, \infty])) \\
    & \qquad\qquad + N \mu (C \cap g^{-1} (]N, \infty])).
  \end{align*}
  The last term is the area of the bottommost band.

  For each $k$, the crescent
  $C' \eqdef C \cap g^{-1} (]k/2^N, \infty]) \diff g^{-1} (](k+1)/2^N,
  \infty])$ is disjoint from $g^{-1} (](k+1)/2^N, \infty])$, so by
  assumption, $\nu (C') \leq \frac {k+1} {2^N} \cdot \mu (C')$.
  Therefore:
  \begin{align*}
    (g \cdot \mu) (C)
    & \geq \sum_{k=1}^{N2^N-1} \frac k {k+1} \nu (C
      \cap g^{-1} (]k/2^N, \infty]) \diff g^{-1} (](k+1)/2^N,
      \infty])) \\
    & \qquad\qquad + N \mu (C \cap g^{-1} (]N, \infty])).
  \end{align*}
  Keeping only the terms from the summation with $k \geq N$ and
  observing that $\frac k {k+1} \geq \frac N {N+1}$ for all such $k$,
  \begin{align*}
    (g \cdot \mu) (C)
    & \geq \sum_{k=N}^{N2^N-1} \frac N {N+1} \nu (C
      \cap g^{-1} (]k/2^N, \infty]) \diff g^{-1} (](k+1)/2^N,
      \infty])) \\
    & \qquad\qquad + N \mu (C \cap g^{-1} (]N, \infty])) \\
    & = \frac N {N+1} \nu (C \cap g^{-1} (]N/2^N, \infty]) \diff
      g^{-1} (]N, \infty])) \\
    & \qquad\qquad + N \mu (C \cap g^{-1} (]N, \infty])) \\
    & = \frac N {N+1} \nu (C \cap g^{-1} (]N/2^N, \infty])) \\
    & \qquad\qquad
      + N \left(\mu (C \cap g^{-1} (]N, \infty])) - \frac 1 {N+1} \nu (C \cap g^{-1} (]N, \infty]))\right).
  \end{align*}
  Let $U_N \eqdef g^{-1} (]N/2^N, \infty])$ for every $N \in \nat$: we
  have just proved $(ii)$.  The family ${(U_N)}_{N \in \nat}$ is
  directed: given any $i, j \in \nat$, there is an $N \in \nat$ such
  that $N/2^N \leq i/2^i, j/2^j$ because $N/2^N$ tends to $0$ as $N$
  tends to $\infty$; and then $U_N$ contains both $U_i$ and $U_j$.
  % (One can also check that the sets $U_N$ form a monotone sequence
  % when $N \geq 2$.)

  Finally, $\dcup_{N \in \nat} U_N = g^{-1} (]0, \infty])$.  Let $C$
  be the crescent $X \diff \dcup_{N \in \nat} U_N$.  This is disjoint
  from $g^{-1} (]r, \infty])$ for every $r \in \D \cap \Rp$, so $\nu
  (C) \leq r \cdot \mu (C)$ for every $r \in \D \cap \Rp$ by
  assumption.  As a consequence, $\nu (C)=0$, and this is $(i)$.
\end{proof}

The role of absolute continuity % and of strong continuity
is as follows.
\begin{lemma}
  \label{lemma:gmu:above:abscont}
  Let $(X, \Latt)$ be a Pervin space, and $\nu$ and $\mu$ be two
  bounded valuations on $(X, \Latt)$.  Let ${(U_N)}_{N \in \nat}$ be a
  countable family of elements of $\Latt$.
  % such that $\nu (X \diff \dcup_{N \in \nat} U_N)=0$.
  If $\nu$ is absolutely continuous with respect to $\mu$, then for
  every $U \in \Latt$, for every $\epsilon > 0$, there is an
  $N_0 \in \nat$ such that for every $N \geq N_0$,
  $N (\mu (U \cap U_N) - \frac 1 {N+1} \nu (U \cap U_N)) \geq
  -\epsilon$.
  % The same hold with $C \in \mathcal A (\Latt)$ instead
  % of $U \in \Latt$ if $\nu$ is strongly continuous with respect to
  % $\mu$.
  % If $\nu$ is absolutely continuous with respect to $\mu$, then for
  % every $U \in \Latt$,
  % $\sup_{N \in \nat} (N (\mu (U \cap U_N) - \frac 1 {N+1} \nu (U
  % \cap
  % U_N))) \geq 0$.
\end{lemma}
\begin{proof}
  % We only deal with absolute continuity, the case of strong continuity
  % is similar.  
  Let us fix an arbitrary $\epsilon > 0$. Using
  Remark~\ref{rem:abscont:bounded}, since $\nu$ and $\mu$ are bounded,
  we can find $\eta > 0$ such that for every $V \in \Latt$ such that
  $\mu (V) < \eta$, $\nu (V) < \epsilon$.  Since $\nu$ is bounded once
  again, there is an $N_0 \in \nat$ such that $N_0 \eta \geq \nu (X)$.
  For every $N \geq N_0$, either $\mu (U \cap U_N) < \eta$, in which
  case $\nu (U \cap U_N) - \epsilon < 0 \leq N \mu (U \cap U_N)$, or
  $\mu (U \cap U_N) \geq \eta$, in which case
  $N \mu (U \cap U_N) \geq N_0 \eta \geq \nu (X) \geq \nu (U \cap U_N)
  - \epsilon$.  Whatever the alternative, we have
  $N \mu (U \cap U_N) - \nu (U \cap U_N) \geq -\epsilon$, and
  therefore
  $N (\mu (U \cap U_N) - \frac 1 {N+1} \nu (U \cap U_N)) \geq
  -\epsilon$, for every $N \geq N_0$.% Hence
  % $\sup_{N \in \nat} (N (\mu (U \cap U_N) - \frac 1 {N+1} \nu (U \cap
  % U_N))) \geq -\epsilon$, and we conclude since $\epsilon > 0$ is
  % arbitrary.
\end{proof}

The following is the only place in this section where we need our
valuations to be $\omega$-continuous.
\begin{lemma}
  \label{lemma:gmu:above:cont}
  Let $\nu$ be an $\omega$-continuous bounded valuation on an
  $\omega$-topological space $(X, \Latt)$, and let
  ${(U_N)}_{N \in \nat}$ be a countable directed family of elements of
  $\Latt$ such that $\nu (X \diff \dcup_{N \in \nat} U_N)=0$.  For
  every $C \in \mathcal A (\Latt)$,
  $\dsup_{N \in \nat} \frac N {N+1} \nu (C \cap U_N) \geq \nu (C)$.
\end{lemma}
\begin{proof}
  Let $U_\infty \eqdef \dcup_{N \in \nat} U_N$.  For every
  $C \in \mathcal A (\Latt)$, the family
  ${(\nu (C \cap U_N))}_{N \in \nat}$ is directed.  This is because
  ${(U_N)}_{N \in \nat}$ is directed and
  $U \in \Latt \mapsto \nu (C \cap U)$ is monotonic.  Indeed,, if
  $U \subseteq V$, then
  $\nu (C \cap V) = \nu (C \cap U) + \nu (C \cap (V \diff U)) \geq \nu
  (C \cap U)$.

  We claim that
  $\dsup_{N \in \nat} \nu (C \cap U_N) \geq \nu (C \cap U_\infty)$ for
  every $C \in \mathcal A (\Latt)$.  (The equality follows by
  monotonicity of $U \mapsto \nu (C \cap U)$.)  By additivity of
  $\nu$, and since $+$ is Scott-continuous, it is enough to show this
  when $C$ is a crescent, say $U' \diff V'$, where $U', V' \in \Latt$
  and $V' \subseteq U'$.  For every $\epsilon > 0$, there is an
  $N \in \nat$ such that
  $\nu (U' \cap U_N) \geq \nu (U' \cap U_\infty) - \epsilon$, since
  $\nu$ is $\omega$-continuous.  Since $\nu$ is monotonic,
  $\nu (V' \cap U_N) \leq \nu (V' \cap U_\infty)$, and therefore
  $\nu (C \cap U_N) = \nu (U' \cap U_N) - \nu (V' \cap U_N) \geq \nu
  (U' \cap U_\infty) - \nu (V' \cap U_\infty) - \epsilon = \nu (C \cap
  U_\infty) - \epsilon$.

  Now, since $\nu (X \diff U_\infty)=0$, we have
  $\nu (C \cap U_\infty) = \nu (C)$.  (Formally,
  $\nu (C \diff U_\infty) \leq \nu (X \diff U_\infty)=0$, and then
  $\nu (C) = \nu (C \cap U_\infty) + \nu (C \diff U_\infty) = \nu (C
  \cap U_\infty)$.)  Therefore
  $\dsup_{N \in \nat} \nu (C \cap U_N) \geq \nu (C)$.  Since
  multiplication is Scott-continuous on $\creal$ and
  $\dsup_{N \in \nat} \frac N {N+1}=1$, we conclude.
\end{proof}

\begin{corollary}
  \label{corl:gmu:above}
  Let $\mu$ and $\nu$ be two bounded $\omega$-continuous valuations on
  an $\omega$-topological space $(X, \Latt)$.  If $\nu$ is absolutely
  continuous with respect to $\mu$ and if for every non-negative
  dyadic number $r \in \D \cap \Rp$, for every crescent $C$ disjoint
  from $g^{-1} (]r, \infty])$, $\nu (C) \leq r \cdot \mu (C)$, then
  $g \cdot \mu \geq \nu$ on $(X, \Latt)$.
\end{corollary}
\begin{proof}
  Let $U_N$ be as in Lemma~\ref{lemma:gmu:above}.  For every
  $U \in \mathcal \Latt$, for every $N \in \nat$, $(g \cdot \mu) (U)$
  is larger than or equal to the sum of
  $\frac N {N+1} \nu (U \cap U_N)$ and of
  $N (\mu (U \cap U_N) - \frac 1 {N+1} \nu (U \cap U_N))$.  For every
  $\epsilon > 0$, the latter is larger than or equal to $-\epsilon$
  for $N$ large enough by Lemma~\ref{lemma:gmu:above:abscont}, and the
  former is larger than or equal to $\nu (U) - \epsilon$ for $N$ large
  enough by Lemma~\ref{lemma:gmu:above:cont}.  Hence
  $(g \cdot \mu) (U) \geq \nu (U) - \epsilon$.  We conclude since
  $\epsilon > 0$ is arbitrary.
\end{proof}

We now go beyond bounded valuations, and on to $\sigma$-finite valuations.
% As with measures, we say that an $\omega$-continuous valuation $\nu$
% on an $\omega$-topological space $(X, \Latt)$ is
% \emph{$\sigma$-finite} if and only if there is a countable family of
% sets $E_n \in \Latt$, $n \in \nat$, such that
% $\bigcup_{n \in \nat} E_n=X$ and $\nu (E_n) < \infty$ for each
% $n \in \nat$.  Replacing $E_n$ by $\bigcup_{k=0}^n E_k$ if necessary,
% we may assume that ${(E_n)}_{n \in \nat}$ is a monotone sequence.
\begin{lemma}
  \label{lemma:sigmafin:two}
  Let $\mu$ and $\nu$ be two $\omega$-continuous valuations on an
  $\omega$-topological space $(X, \Latt)$.  If both $\nu$ and $\mu$
  are $\sigma$-finite, there is a monotone sequence ${(E_n)}_{n \in
    \nat}$ of elements of $\Latt$ such that $\dcup_{n \in \nat} E_n=X$
  and $\nu (E_n), \mu (E_n) < \infty$ for each $n \in \nat$.
\end{lemma}
\begin{proof}
  Let ${(F_n)}_{n \in \nat}$ be a monotone sequence of elements of
  $\Latt$ such that $\nu (F_n) < \infty$ and
  $\dcup_{n \in \nat} F_n=X$, and let ${(G_n)}_{n \in \nat}$ play the
  same r\^ole with $\mu$.  Then let $E_n \eqdef F_n \cap G_n$ for each
  $n \in \nat$.
\end{proof}
We will call any monotone sequence ${(E_n)}_{n \in \nat}$ satisfying
the conclusion of Lemma~\ref{lemma:sigmafin:two} a \emph{witness} of
the joint $\sigma$-finiteness of $\nu$ and $\mu$.

\begin{theorem}[Existence of density maps]
  \label{thm:nikodym:weak}
  Let $(X, \Latt)$ be an $\omega$-topological space, and $\mu$ and
  $\nu$ be two $\sigma$-finite $\omega$-continuous valuations on
  $(X, \Latt)$.  Let ${(E_n)}$ be any witness of joint
  $\sigma$-finiteness of $\nu$ and $\mu$. Then the following
  properties are equivalent:
  \begin{enumerate}
  \item there is a density function $g \in \Lcont (X, \Latt)$ such
    that $\nu = g \cdot \mu$;
  \item the following two conditions are met:
    \begin{enumerate}
    \item[$(2a)$] $\nu$ is absolutely continuous with respect to $\mu$;
    \item[$(2b)$] for every $n \in \nat$, for every $r \in \Rp$,
      $\nu_{|E_n} - r \cdot \mu_{|E_n}$ has the Hahn decomposition
      property.
    \end{enumerate}
  \end{enumerate}
\end{theorem}
\begin{proof}
  The implication $(1)\limp (2)$ is by Proposition~\ref{prop:abscont}
  and Corollary~\ref{corl:hahn}.

  In the converse direction, let ${(E_n)}_{n \in \nat}$ be as given in
  Lemma~\ref{lemma:sigmafin:two}.  For each $n \in \nat$ and for each
  non-negative rational number $q$, $\nu_{|E_n} - q \cdot \mu_{|E_n}$
  has the Hahn decomposition property, so there is an element
  $U_{nq} \in \Latt$ such that every crescent $C$ included in $U_{nq}$
  satisfies $\nu (C \cap E_n) \geq q \cdot \mu (C \cap E_n)$ and every
  crescent $C$ disjoint from $U_{nq}$ satisfies
  $\nu (C \cap E_n) \leq q \cdot \mu (C \cap E_n)$.

  Since $\Latt$ is an $\omega$-topology,
  $V_q \eqdef \bigcup_{\substack{q' \in \rat, q' \geq q\\n \in \nat}}
  (E_n \cap U_{nq})$ is in $\Latt$ for every $q \in \rat$, $q \geq 0$.
  Moreover, ${(V_q)}_{q \in \rat, q \geq 0}$ forms an antitonic chain:
  if $q \leq q'$ then $V_q \supseteq V_{q'}$.

  Given $n \in \nat$ and $q \in \rat$, $q \geq 0$, we claim that for
  every crescent $C \subseteq V_q$,
  $\nu (C \cap E_n) \geq q \cdot \mu (C \cap E_n)$, and that every
  crescent $C$ disjoint from $V_q$ satisfies
  $\nu (C \cap E_n) \leq q \cdot \mu (C \cap E_n)$.  The second
  property is clear: if $C$ is disjoint from $V_q$, then it is
  disjoint from $E_n \cap U_{nq}$, so $C \cap E_n$ is a crescent
  disjoint from $U_{nq}$, whence
  $\nu ((C \cap E_n) \cap E_n) \leq q \cdot \mu ((C \cap E_n) \cap
  E_n)$.  For the first property, where $C \subseteq V_q$, let us
  write $C$ as $U \diff V$ where $U, V \in \Latt$.  We enumerate the
  rational numbers larger than or equal to $q$ as
  ${(q_m)}_{m \in \nat}$.  Since $C \subseteq V_q$,
  $\nu (C \cap E_n) = \nu (C \cap E_n \cap V_q)$.  Now
  $E_n \cap V_q = \dcup_{p, p' \geq n} W_{pp'}$, where
  $W_{pp'} \eqdef \bigcup_{\substack{0 \leq j\leq p\\0\leq k\leq p'}}
  (E_j \cap E_n \cap U_{jq_k})$.  Therefore
  $\nu (C \cap E_n) = \nu (C \cap E_n \cap V_q) = \nu (U \cap E_n \cap
  V_q \diff V) = \nu ((U \cap E_n \cap V_q) \cup V) - \nu (V) =
  \dsup_{p, p' \in \nat} \nu ((U \cap W_{pp'}) \diff V) - \nu (V) =
  \dsup_{p, p' \in \nat} \nu (C \cap W_{pp'})$.  Similarly,
  $\mu (C \cap E_n) = \dsup_{p, p' \in \nat} \mu (C \cap W_{pp'})$.
  We can write $W_{pp'}$ as the finite disjoint union of crescents
  $C_{jk}$ with $0 \leq j\leq p$ and $0 \leq k\leq p'$, where
  $C_{jk} \eqdef (E_j \cap E_n \cap U_{jq_k}) \diff
  \bigcup_{\substack{0 \leq j'\leq j\\0\leq k'\leq k\\ (j',k') \neq
      (j,k)}} (E_{j'} \cap E_n \cap U_{j'q_{k'}})$.  Then
  $C_{jk} \subseteq U_{jq_k}$, hence also
  $C \cap C_{jk} \subseteq U_{jq_k}$, so
  $\nu (C \cap C_{jk} \cap E_j) \geq q_k \cdot \mu (C \cap C_{jk} \cap
  E_j)$.  Since $C_{jk} \subseteq E_j$, this simplifies to
  $\nu (C \cap C_{jk}) \geq q_k \cdot \mu (C \cap C_{jk})$.  Then
  $\nu (C \cap W_{pp'}) = \sum_{\substack{0\leq j\leq p\\0\leq k\leq
      p'}} \nu (C \cap C_{jk}) \geq \sum_{\substack{0\leq j\leq
      p\\0\leq k\leq p'}} q_k \cdot \mu (C \cap C_{jk})$.  Since
  $q_k \geq q$ for every $k$, this is larger than or equal to
  $q \cdot \sum_{\substack{0\leq j\leq p\\0\leq k\leq p'}} \mu (C \cap
  C_{jk}) = q \cdot \mu (C \cap W_{pp'})$.  Taking suprema over
  $p, p' \in \nat$, we obtain that
  $\nu (C \cap E_n) \geq q \cdot \mu (C \cap E_n)$, as desired.
  
  We define $g (x)$ as
  $\dsup \{t \in \Rp \mid \exists q \in \rat, q > t \text{ and } x \in
  V_q\}$.  Then $g (x) > t$ if and only if $x \in V_q$ for some
  $q \in \rat$, $q > t$.  Hence
  $g^{-1} (]t, \infty]) = \bigcup_{q \in \rat, q > t} V_q$, which is
  in $\Latt$ since $\Latt$ is an $\omega$-topology.
  % (That union is
  % even a countable chain, but not a monotone sequence.)
  Therefore $g$ is in $\Lcont (X, \Latt)$.

  Let us fix $n \in \nat$.  For every non-negative dyadic number
  $r \in \D \cap \Rp$,
  $g^{-1} (]r, \infty]) = \bigcup_{q \in \rat, q > r} V_q \subseteq
  V_r$, so for every crescent $C \subseteq g^{-1} (]r, \infty])$,
  $\nu (C \cap E_n) \geq r \cdot \mu (C \cap E_n)$.  By
  Lemma~\ref{lemma:gmu:below}, $\nu_{|E_n} \geq g \cdot \mu_{|E_n}$.
  For every crescent $C$ disjoint from $g^{-1} (]r, \infty])$, $C$ is
  disjoint from every $V_q$ with $q > r$, so
  $\nu (C \cap E_n) \leq q \cdot \mu (C \cap E_n)$ for every rational
  $q > r$; therefore $\nu_{|E_n} (C) \leq r \cdot \mu_{|E_n} (C)$, and
  by Corollary~\ref{corl:gmu:above},
  $g \cdot \mu_{|E_n} \geq \nu_{|E_n}$.

  It follows that $\nu_{|E_n} = g \cdot \mu_{|E_n}$ for every
  $n \in \nat$.  Then, using the fact that
  $X = \dcup_{n \in \nat} E_n$ and the $\omega$-continuity of $\nu$,
  for every $U \in \Latt$,
  $\nu (U) = \dsup_{n \in \nat} \nu_{|E_n} (U) = \dsup_{n \in \nat} (g
  \cdot \mu_{|E_n}) (U) = \dsup_{n \in \nat} \int_0^\infty \mu (U \cap
  E_n \cap g^{-1} (]t, \infty])) \;dt$ (by Lemma~\ref{lemma:gmu:int}),
  and this is equal to
  $\int_0^\infty \mu (E \cap g^{-1} (]t, \infty])) \;dt$ by
  Fact~\ref{fact:R:cont} and the $\omega$-continuity of $\mu$, namely
  to $(g \cdot \mu) (U)$.  Therefore $\nu = g \cdot \mu$.
\end{proof}

\begin{remark}
  \label{rem:Hahn}
  In the special case where $\Latt$ is not only an $\omega$-topology,
  but is also closed under complements, namely when $\Latt$ is a
  $\sigma$-algebra, we have seen that $\omega$-continuous valuations
  and measures are the same thing.  Then, for every $n \in \nat$ and
  for every $r \in \Rp$, $\nu_{|E_n} - r \cdot \mu_{|E_n}$ is a signed
  measure.  The \emph{Hahn decomposition theorem}
  \cite[Theorem~32.1]{Billingsley:probmes} states that every signed
  measure has the Hahn decomposition property, and therefore property
  $(2b)$ is simply true in the case of measures.
  % Billingsley calls 'signed measures' additive, but beware that
  % additivity is also for infinite sequences of pairwise disjoint
  % sets in his terms.
  % Additionally, property $(2a)$, that $\nu$ is absolutely continuous
  % with respect to $\mu$, is equivalent to the familiar property ``for
  % every $U \in \Latt$ such that $\mu (U)=0$, $\nu(U)=0$''.  One can
  % find a proof of this in the case where $\nu$ is bounded in
  % \cite[page~422]{Billingsley:probmes}.  We prove it in the general
  % case as follows.
  Hence Theorem~\ref{thm:nikodym:weak} implies the classical
  Radon-Nikod\'ym theorem.
\end{remark}

\bibliographystyle{abbrv}
%\bibliography{rn}

\end{document}